\documentclass{amsart}
\usepackage{amssymb}
\usepackage{subfig}
\usepackage{graphicx}

\newtheorem{theorem}{Theorem}[section]
\newtheorem{lemma}[theorem]{Lemma}

\theoremstyle{definition}
\newtheorem{definition}[theorem]{Definition}
\newtheorem{example}[theorem]{Example}

\theoremstyle{remark}
\newtheorem{remark}[theorem]{Remark}

\numberwithin{equation}{section}

\newcommand\rd{{\mathrm{d}}} 

\newcommand\re{{\mathrm{e}}}

\def\bI{\boldsymbol{I}}

\def\b0{\boldsymbol{0}}

\def\bv{\boldsymbol{v}}

\def\btau{\boldsymbol{\tau}}
\def\bsigma{\boldsymbol{\sigma}}

\def\bepsilon{\boldsymbol{\epsilon}}

\def\R{\mathbb R}
\def\mS{\mathbb S}
\def\mM{\mathbb M}

\def\cA{\mathcal {A}}

\def\cN{\mathcal {N}}

\def\cT{\mathcal{T}}

\def\cP{\mathcal{P}}

\def\cF{\mathcal{F}}

\newcommand{\curl}{{\rm curl}}

\newcommand{\tr}{{\rm tr}}

\newcommand{\supp}{{\rm supp}}
\renewcommand{\div}{{\rm div}}

\usepackage{cases}
\usepackage{subfig}
\usepackage{slashbox}

\begin{document}

\title[Hybridized Mixed methods for Elasticity and Multilevel
Solvers]{New Hybridized Mixed Methods for Linear Elasticity and
Optimal Multilevel Solvers}


\author[S. Gong]{Shihua Gong}
\address{Beijing International Center for Mathematical Research,
Peking University, Beijing 100871, P. R. China}
\curraddr{}
\email{gongshihua@pku.edu.cn}
\thanks{The work of the first and third authors was supported in part
by National Natural Science Foundation of China (NSFC) (Grant No.
91430215, 41390452) and by Beijing International Center for
Mathematical Research of Peking University, China.}

\author[S. Wu]{Shuonan Wu}
\address{Department of Mathematics, Pennsylvania State University,
University Park, PA, 16802, USA}
\curraddr{}
\email{sxw58@psu.edu}
\thanks{The work of the second and third authors was supported in part
by the DOE Grant DE-SC0009249 as part of the Collaboratory on
Mathematics for Mesoscopic Modeling of Materials and by DOE Grant
DE-SC0014400 and NSF Grant DMS-1522615.}

\author[J. Xu]{Jinchao Xu}
\address{Department of Mathematics, Pennsylvania State University,
University Park, PA 16802, USA}
\curraddr{}
\email{xu@math.psu.edu}



\dedicatory{}

\begin{abstract}
In this paper, we present a family of new mixed finite element methods
for linear elasticity for both spatial dimensions $n=2,3$, which
yields a conforming and strongly symmetric approximation for stress.
Applying $\mathcal{P}_{k+1}-\mathcal{P}_k$ as the local
approximation for the stress and displacement, the mixed methods
achieve the optimal order of convergence for both the stress and
displacement when $k \ge n$.  For the lower order case $(n-2\le k<n)$,
the stability and convergence still hold on some special grids. The
proposed mixed methods are efficiently implemented by hybridization,
which imposes the inter-element normal continuity of the stress by a
Lagrange multiplier. Then, we develop and analyze multilevel solvers
for the Schur complement of the hybridized system in the two
dimensional case. Provided that no nearly singular vertex on the
grids, the proposed solvers are proved to be uniformly convergent with
respect to both the grid size and Poisson's ratio. Numerical
experiments are provided to validate our theoretical results.
\end{abstract}

\maketitle

\section{Introduction} \label{sec:intro}

The mixed finite element methods are popular in solid mechanics since
they avoid locking and provide a straightforward approximation for
stress. The conforming mixed methods based on the classical
Hellinger-Reissner variational formulation requires finite element
space for the stress in $H(\div;\mS)$, the space of symmetric
matrix-valued fields, which are square integrable with square
integrable divergence. In the meantime, the discrete space for the
stress must be compatible with that for the displacement, which is a
subspace of the vector-valued $L^2$ space. However, the construction of
such stable pairs using polynomial shape functions is very challenging. 

To overcome this difficulty, the earliest works adopted composite
element techniques (cf. \cite{johnson1978some, arnold1984family}). The
composite element methods approximate the displacement in one grid
while approximating the stress in the refined grid.  Due to the
difficulties in keeping the symmetry and conformity at the same time,
some compromised methods that relax one of the two requirements have
been developed. The first category of such methods (cf.
\cite{amara1979equilibrium, morley1989family, qiu2009mixed,
arnold2007mixed, boffi2009reduced, cockburn2010new,
guzman2010unified}) weakly imposes stress symmetry, while
maintaining exact $H(\div)$ conformity.  These methods introduce the
Lagrange multiplier, approximating the non-symmetric part of the
displacement gradient while enforcing stress symmetry weakly. The
second category of such methods (cf. \cite{arnold2003nonconforming,
awanou2009rotated, hu2007lower, man2009lower, yi2005nonconforming,
yi2006new, gong2015lowest}) relaxes the conformity constraints while 
keeping the symmetry strongly. 

In \cite{arnold2002mixed}, Arnold and Winther proposed the first
family of mixed finite element methods in two dimension (2D), which yields the symmetric
and conforming approximation for the stress. Since then, many stable mixed
elements have been constructed, see \cite{arnold2005rectangular,
arnold2008finite, adams2005mixed}.
However,  the shape function spaces of these elements, using
incomplete polynomials, are quite complicated.  In \cite{hu2014family,
hu2015family}, Hu and Zhang constructed a family of mixed finite
elements with conforming and symmetric stress approximation in a
unified fashion on simplex grids for spatial dimension $n=2, 3$. The
degrees of the polynomials to approximate the stress and displacement
match reasonably and naturally, by which these elements also achieve
the optimal order of convergence.  The generalizations or variants of
Hu-Zhang's finite elements can be found in \cite{hu2015finite,
hu2015new, hu2016finite}. 

Both families of the conforming elements above are subject to
continuity constraints at the element vertices, which is not natural
for $H(\div)$ conformity and prohibits techniques like
hybridization that are usually available for the mixed method. 
One feature of our methods is to relax the continuity at the element
vertices using the full $C^{\div}$--$\cP_{k+1}$ space for the
stress
$$
\Sigma_{h,k+1} = \{ \btau \in H(\div, \Omega; \mS)~|~ \btau|_{K} \in
\cP_{k+1}(K; \mS) \quad \forall K \in \cT_h\}.
$$
Taking the full $C^{-1}$--$\cP_{k}$ vector-valued space $V_{h,k}$ for
the displacement, the stability of $\Sigma_{h,k+1}-V_{h,k}$ follows
directly from the results of \cite{hu2014family, hu2015family,
hu2015finite} when $k \ge n$. On some special grids, we can still
prove the stability for the lower order pairs when $n-2 \le k < n $. In the
2D case, it is feasible to construct nodal basis functions for
$\Sigma_{h,k+1}$ by geometric analysis at the vertices (cf.
    \cite{morgan1975nodal}).  In the 3D case, however, it is complicated
to deal with nodal basis functions associated with the vertices or
edges. In any case, the dimension of $\Sigma_{h,k+1}$ therefore
depends on the singular vertices (cf.  \cite{morgan1975nodal}) or
singular edges of the grids.

Instead of constructing basis functions for $\Sigma_{h,k+1}$, we
implement it by hybridization (cf. \cite{arnold1985mixed,
cockburn2009unified}).  In other words, we remove the inter-element
continuity of stress and enforce it by the Lagrange multiplier---the
piecewise discontinuous polynomial space of degree $k+1$ defined on
the edges or faces. The stress and displacement can be eliminated
locally in the hybridized mixed system, which results in a linear
system solely for the Lagrange multiplier. The resulting 
multiplier system may have a nontrivial kernel due to the
singular vertices or singular edges on the grids but leads to a unique
solution of the stress and displacement.  Related works on
hybridizable methods for elasticity can be found in
\cite{soon2009hybridizable, gopalakrishnan2011symmetric, qiu2013hdg}.
In \cite{gopalakrishnan2011symmetric}, a family of nonconforming and
hybridizable elements on simplicial grids was developed in both 2D and
3D cases. The hybridizable discontinuous Galerkin (HDG) methods for the linear elasticity were studied in
\cite{soon2009hybridizable, qiu2013hdg}.


Another feature of our methods is to develop robust iterative solvers
for the Schur complement of the hybridized mixed system in the 2D case,
provided that there is no nearly singular vertex on the grids. The
iterative solvers for the hybridized mixed method for the diffusion
problem were studied in \cite{gopalakrishnan2003schwarz,
gopalakrishnan2009convergent, cockburn2013multigrid, li2016bpx,
li2016analysis}. Although the methodologies in dealing with the
non-nested multilevel finite element spaces and the non-inherited
bilinear forms were discussed in these papers for the diffusion
problem, two essential distinctions exist for the linear elasticity:
(i) some local estimates do not hold on each element, but on the
element patch, and (ii) the condition number of the multiplier system
depends not only on the grid size but also on Poisson's ratio.

To overcome these difficulties, we first establish some local
estimates on the element patches by characterizing the inter-element
jump of piecewise discontinuous symmetric-matrix-valued polynomials
(see Lemma \ref{lm:stability-local-jump} and \ref{lm:stableJ}).
We then propose an equivalent norm to the energy norm associated with
the multiplier system, which indicates that the multiplier system
holds a similar structure with that of the stable discretization
($P_2$--$P_0$) for the elastic primal formulation (cf.
\cite{schoberl1999multigrid}). Thus, capturing the rigid-body
motion mode and the weak divergence-free mode simultaneously is the
key to developing robust iterative solvers with respect to both the
grid size and Poisson's ratio.
  
The rest of the paper is organized as follows. In the next section, we
introduce our mixed finite element methods and prove their stability
and convergence. In Section \ref{sec:hybrid}, we present the
hybridization of the mixed finite element method.  We also
characterize the kernel of the hybridized mixed system and develop
some tools to estimate the norms. In Section \ref{sec:schwarz}, we
focus on the iterative solvers for the multiplier system.  We provide
some numerical results in Section \ref{sec:numerical} and give some
concluding remarks in Section \ref{sec:concluding}.  Finally, some
technical results can be found in the appendix.

\section{Mixed Methods}
\label{sec:mixed}

In this paper, we consider the following linear elasticity problem
with Dirichlet boundary condition  
\begin{equation} \label{equ:elasticity}
\left\{
\begin{aligned}
\cA \bsigma - \bepsilon(u) &= 0 \quad \text{in~}\Omega, \\
\div \bsigma &= f \quad \text{in~} \Omega, \\
u &= 0 \quad \text{on~} \partial\Omega, 
\end{aligned}
\right.
\end{equation}
where $\Omega$ is a polygonal domain in $\R^n ~(n=2, 3)$. The
displacement and stress are denoted by $u: \Omega \mapsto \R^n$ and
$\bsigma: \Omega \mapsto \mS$, respectively. Here, $\mS$ represents
the space of real symmetric matrices of order $n \times n$. The
compliance tensor $\cA: \mS \mapsto \mS$ is defined as
\begin{equation} \label{eq:compliance}
\mathcal{A} \bsigma := \frac{1}{2\tilde{\mu}} \left(\bsigma -
    \frac{\tilde{\lambda}}{2\tilde{\mu} +
n \tilde{\lambda}} \tr(\bsigma)\bI \right),
\end{equation}
where $\tilde{\mu}, \tilde{\lambda} $ are the Lam\'{e} constants.
Clearly, $\mathcal{A}$ is bounded and symmetric positive definite.
The linearized strain tensor is denoted by $\bepsilon(u) = (\nabla u +
(\nabla u)^T)/2$.

\subsection{Preliminaries}
Let $\cT_h$ be a family of quasi-uniform triangulations (cf.
\cite{brenner2007mathematical}) of $\Omega$.  Let $h_K$ be the
diameter of element $K \in \cT_h$, and $h = \max_K h_K$ be the grid
diameter of $\cT_h$.  For any $K \in \cT_h$, the set of all elements
that share  vertex with $K$ is denoted by $\omega_K$.  The sets of all
faces and nodes of $\cT_h$ are denoted by $\cF_h$ and $\cN_h$,
respectively.  Moreover, $\cF_h$ can be divided into two subsets: the
boundary faces set $\cF_h^{\partial} = \cF_h \cap \partial\Omega$ and
the interior faces set $\cF_h^i = \cF_h \setminus \cF_h^{\partial}$.
The unit normal vector with respect to the face $F$ is represented by
$\nu_F$.

Let $F \in \cF_h^i$ be the common face of two elements $K^+$ and
$K^-$, and $\nu_F^+$ and $\nu_F^-$ be the unit outward normal vectors
on $F$ with respect to $K^+$ and $K^-$, respectively. Then, we define
the jump $[\cdot]$ on $F \in \cF_h^i$ for $\btau$ by: 
\begin{equation*} \label{equ:jump}
[\btau]_F := \btau_{K^+} \nu_F^+ + \btau_{K^-} \nu_F^-.
\end{equation*}
For $F\in \cF_h^{\partial }$, we define $[\btau]_F := \btau \nu$,
where $\nu$ is the unit outer normal along $\partial \Omega$.

Our notation for the inner products is standard (cf.
\cite{brenner2007mathematical}): For $u,v \in L^2(D)$, we
write $(u,v)_D=\int_D uv ~\rd x$ if $D$ is a subdomain of $\R^n$, and
$\langle u, v \rangle_D = \int_D uv ~\rd s$ if $D$ is a subdomain of
$\R^{n-1}$. We neglect the subscript $D$ if $D=\Omega$.  To emphasize
the mesh-dependent nature of certain integrals, for $\widetilde{\cT}_h
\subset \cT_h$ and $\widetilde{\cF}_h \subset \cF_h$, we define  
$$
(u_h, v_h)_{\widetilde{\cT}_h} := \sum_{K\in \widetilde{\cT}_h}(u_h,
v_h)_K \quad \text{and} \quad \langle \lambda_h, \mu_h
\rangle_{\widetilde{\cF}_h} := \sum_{F \in \widetilde{\cF}_h} \langle
\lambda_h, \mu_h \rangle_{F}, 
$$
where $u_h, v_h$ and $\mu_h, \lambda_h$ are defined on
$\widetilde{\cT}_h$ and $\widetilde{\cF}_h$, respectively.

Throughout this paper, we shall use letter $C$ to denote a generic
positive constant independent of $h$ and the material parameters. Note
that $C$ may stand for different values at its various occurrences.
The notation $x \lesssim y$ means $x \leq Cy$ and $x \simeq y$ means
$x \lesssim y \lesssim x$.

The mixed formulation of \eqref{equ:elasticity} is to find $(\bsigma,
u) \in \Sigma \times V := H(\div, \Omega; \mS) \times L^2(\Omega;
\R^n)$ such that 
\begin{equation} \label{equ:mixed-formulation}
\left\{
\begin{aligned}
&(\mathcal{A} \bsigma, \btau) + (\div \btau, u) &=&~ 0 \quad
&\forall \btau \in \Sigma, \\
&(\div \bsigma, v) &=&~ (f, v) \quad &\forall v \in V.
\end{aligned}
\right.
\end{equation}
Here, $H(\div, \Omega; \mS)$ consists of square-integrable symmetric
matrix fields with square-integrable divergence, and $L^2(\Omega;
\R^n)$ is the space of vector-valued functions that are
square integrable with the standard $L^2$ norm.  The corresponding
$H(\div)$ norm is defined by 
\[
\|\btau\|_{H(\div)}^2 := \|\btau\|_{0}^2 + \|\div \btau\|_{0}^2 \qquad
\forall \btau \in H(\div, \Omega; \mS).
\]

We take the discrete stress space as the full $C^{\div}$--$\cP_{k+1}$
space
\begin{equation}\label{eq:def-stress-space}
\Sigma_{h,k+1} := \{ \btau \in H(\div, \Omega; \mS)~|~ \btau|_{K} \in
\cP_{k+1}(K; \mS)\quad \forall K \in \cT_h\},
\end{equation}
and take the discrete displacement space as the full
$C^{-1}$--$\cP_{k}$ space
\begin{equation} \label{eq:displacement-space}
V_{h,k} := \{v\in L^2(\Omega;\R^n)~|~v|_{K}\in \cP_{k}(K,\R^n) \quad
\forall K \in \cT_h\}.
\end{equation}
Then, the mixed finite element approximation of the elastic problem
\eqref{equ:mixed-formulation} reads: Find $(\bsigma_h, u_h )\in
\Sigma_{h,k+1} \times V _{h,k}$ such that
\begin{equation}\label{equ:discrete:formulation}\left\{
\begin{aligned}
&(\cA \bsigma_h, \btau_h) + (u_h, \div \btau_h) &=&~ 0 \quad &\forall
\btau_h \in \Sigma_{h,k+1},\\
&( \div \bsigma_h,  v_h)  &=&~ (f, v_h) \quad &\forall v_h \in V
_{h,k}. ~\quad\\
\end{aligned}\right.
\end{equation}

\subsection{Stability and Convergence}

The convergence of the finite element solution follows from the
stability and the standard approximation property. First, we consider
the stability of the discrete problem
\eqref{equ:discrete:formulation}, which follows from two conditions by
the standard theory of mixed finite element methods (cf.
\cite{brezzi2012mixed}).
\begin{enumerate}
\item K-ellipticity: There exists a constant $\alpha > 0$, independent
of the grid size, such that 
\begin{equation}\label{eq:K-ellipticity}
(\mathcal{A} \btau_h, \btau_h) \ge \alpha \|\btau_h\|^2_{H(\div)}
\qquad \forall \btau_h \in Z_h,
\end{equation}
where $Z_h := \{\btau_h \in \Sigma_{h,k+1}~ |~ (\div \btau_h, v_h)=0
\quad \forall v_h \in V_{h,k}\}= \{\btau_h \in \Sigma_{h,k+1}~ |~ \div
\btau_h=0\}$.
\item Lady\v zenskaja-Babu\v ska-Brezzi (LBB) condition: There exists a constant
$\beta > 0$, independent of the grid size, such that
\begin{equation}\label{eq:infsup}
\inf_{v_h \in V_h}\sup_{\btau_h \in \Sigma_{h,k+1}} \frac{(\div
\btau_h, v_h)} { \|\btau_h\|_{H(\div)} \|v_h\|_{0} } \ge \beta.
 \end{equation}
\end{enumerate}

Since $\div \Sigma_{h,k+1}\subset V _{h,k}$ for any $k\ge 0$, we know
that $Z_h \in \mathrm{ker}(\div)$. Therefore, 
\begin{equation} \label{eq:K-ellipticity-elasticity}
(\cA \btau_h, \btau_h) \ge C\|\btau_h\|^2_{0} =
C\|\btau_h\|^2_{H(\div)} \qquad\forall \btau_h \in Z_h,
\end{equation}
as the compliance tensor is positive definite. This implies
the K-ellipticity. Note that, pertaining to $\int_{\Omega}
\tr(\btau_h)~\rd x = 0$, the constant $C$ in
\eqref{eq:K-ellipticity-elasticity} is uniform with respect to the
Poisson's' ratio ($\tilde{\nu} :=
\frac{\tilde{\lambda}}{2(\tilde{\lambda} + \tilde{\mu})}$) due
to the following theorem (see Section 9 in \cite{brezzi2012mixed} for
details).

\begin{theorem} \label{thm:incompressible}
Assume that $\bsigma \in H(\div, \Omega; \mS)$ satisfies $\int_\Omega
\tr(\bsigma) = 0 $ and $\div \bsigma = 0$. It holds that
\begin{equation}\label{eq:incompressible}
\|\bsigma\|_{0}^2 \lesssim  2\tilde{\mu} (\mathcal{A} \bsigma,\bsigma).
\end{equation}
\end{theorem} 

Next, we discuss the inf-sup condition under the pure displacement
boundary condition. Similar techniques work for the traction boundary
condition.

\begin{lemma}\label{lm:infsup-high}
When $k\ge n$, for any $v_h\in V_{h,k}$, there exists $\btau_h \in
\Sigma_{h,k+1}$ such that 
\begin{equation}\label{eq:infsup-high}
\div \btau_h = v_h \quad \text{and} \quad \|\btau_h\|_{H(\div)}
\lesssim \|v_h\|_{0}.
\end{equation}
\end{lemma}
\begin{proof}
This is a corollary of \cite{hu2015finite, hu2014family,
hu2015family}, in which a family of finite elements for $H(\div,
\Omega; \mS)$ satisfying \eqref{eq:infsup-high}  is proposed as
\begin{equation} \label{eq:Hu-Zhang}
\begin{aligned}
{\Sigma}_{h,k+1}^{\rm HZ} := \{ \btau & \in H(\div,\Omega;\mS)~|~ \btau
= \btau_c+\btau_b, ~\btau_c\in H^1(\Omega;\mS), \\ &\btau_c|_K \in
\cP_{k+1}(K;\mS),~ \btau_b|_K \in \Sigma_{k+1,b}(K)\quad  \forall K \in
\cT_h\}.
\end{aligned}
\end{equation}
Here, the local conforming $\div$-bubble space $\Sigma_{k+1,b}(K)
:=\{\btau \in \cP_{k+1}(K;\mS)~|~\btau\nu |_{\partial K} =0 \}$.
Hence, the lemma follows from the fact that $\btau_h \in
{\Sigma}^{\rm HZ}_{h,k+1} \subset \Sigma_{h,k+1}$.
\end{proof}

For the lower order case, the inf-sup condition \eqref{eq:infsup-high}
resorts to some known results of the Stokes pair.
When $k\ge n-2$, the Stokes pair $\cP_{k+2}-\cP_{k+1}^{-1}$ can be
proved stable on special grids (cf. \cite{arnold1992quadratic,
zhang2005new}), a popular example of which is the Hsieh-Clough-Tocher (HCT) grid, where each
macro-simplex is divided into $n+1$ sub-simplexes by connecting the
barycenter with the vertices.

\begin{lemma}\label{lm:infsup-low}
When $n-2 \le k < n$, if the Stokes pair
$\cP_{k+2}-\cP_{k+1}^{-1}$ is stable on the grid, then for any $v_h\in
V_{h,k}$, there exists $\btau_h\in \Sigma_{h,k+1}$ such that 
\begin{equation}
\div \btau_h = v_h \quad \text{and} \quad \|\btau_h\|_{H(\div)}
\lesssim \|v_h\|_{0}.
\end{equation}
\end{lemma}
\begin{proof}
We prove the stability by a constructive method (cf.
\cite{arnold2007mixed}). 
In light of the Brezzi-Douglas-Marini (BDM) elements for $H(\div;\R^n)$ (cf.
\cite{brezzi1985two, brezzi2012mixed}), we defined the following space
\begin{equation*}
{\rm BDM}_{k+1}^{n\times n} := \{\btau \in H(\div,\Omega;\mM)~|~
\btau|_K\in \cP_{k+1}(K;\mM) \quad \forall K \in \mathcal{T}_h\},
\end{equation*}
where $\mM$ represents the space of real matrices of order $n\times
n$. The $\div \btau$ here is defined by taking $\div$ on each row of
$\btau$. By the stability of BDM elements, we immediately
know that for any $v_h \in V_h$, there exists a $\tilde{\btau}_{h}\in {\rm
BDM}^{n\times n}_{k+1}$ such that 
\begin{equation*} \label{eq:infsup_bdm}
\div \tilde{\btau}_{h} = v_h \quad \text{and} \quad
\|\tilde{\btau}_{h}\|_{H(\div)} \lesssim \|v_h\|_{0}.
\end{equation*}

With the purpose of symmetrizing $\btau_{h}$, we add a
divergence-free term to $\tilde{\btau}_{h}$ to obtain 
$$ 
\btau_h = \tilde{\btau}_{h} + \curl \rho_h, 
$$ 
where $\rho_h$ satisfies
\begin{enumerate}
\item For $n=2$: $\rho_h \in H^1(\Omega;\R^2)$ is a vector-valued
function and $\rho_h|_K\in \cP_{k+2}(K;\R^2) $;
\item For $n=3$: $\rho_h \in H^1(\Omega; \mM)$ is a matrix-valued
function and $\rho_h|_K\in \cP_{k+2}(K;\mM) $.
\end{enumerate}
For the 2D case, the $\curl $ operator is a rotation of the 
operator $\nabla$ (i.e., $\curl =(- \partial_y, \partial_x)$) and applies on
each entry of the vector $\rho_h$. For the 3D case, the $\curl$ operator
applies on each row of the matrix $\rho_h$. By direct calculation,
the symmetry of $\btau_h$ is equivalent to the following equation, 
\begin{equation}\label{eq:infsup-low1}
{\rm skw}(\curl \rho_h) = - {\rm skw}\tilde{\btau}_h,
\end{equation}
where ${\rm skw}\btau := (\btau-\btau^T)/2 $.  For a scalar function
$v$ and a vector-valued function $v=(v_1,v_2,v_3)^T$, we further
define 
\[
\mathrm{Skw}_2(v) := \begin{bmatrix} 0 & v\\ -v &0\end{bmatrix} \quad
\text{and} \quad
\mathrm{Skw}_3 (v) := \begin{bmatrix} 0 & v_3 & - v_2 \\ -v_3 & 0 &
v_1 \\ v_2 & -v_1 & 0\\\end{bmatrix}.
\]
Then, the proof can be divided into the following two cases: 
\begin{enumerate}
\item For $n=2$: From \cite{arnold2006finite}, we have $\mathrm{skw}
(\curl \rho_h) = \frac{1}{2} \mathrm{Skw_2}(\div \rho_h)$. Thus,
\eqref{eq:infsup-low1} can be written as:
\begin{equation}\label{eq:infsup-low2D}
 \div \rho_h = \tilde{\tau}_{h,21} - \tilde{\tau}_{h,12}.
\end{equation}
The stability of Stokes pair $\cP_{k+2}-\cP_{k+1}^{-1}$ then
implies that there exists a $\rho_h \in \{ v\in
H^1(\Omega;\mathbb{R}^2)~|~ v|_K \in \cP_{k+2}(K;\mathbb{R}^2)\}$
satisfying \eqref{eq:infsup-low2D} and 
$$
\|\rho_h\|_1 \lesssim \|\tilde{\tau}_{h,21} -
\tilde{\tau}_{h,12}\|_{0} \leq \|\tilde{\btau}_h\|_0 \lesssim \|v_h\|_0.
$$

\item For $n=3$: From \cite{arnold2006finite}, we have  
$\mathrm{skw}(\curl \rho_h) = -\frac{1}{2}\mathrm{Skw_3 }(\div ~ \Xi
\rho_h)$, where $\Xi$ is an algebraic operator defined as ${\Xi}\rho_h
= \rho_h^T - \tr(\rho_h) \bI$. Denoting $\eta_h = \Xi \rho_h$, it is
obvious that $\rho_h = \Xi^{-1} \eta_h =  \eta_h^T -
\frac{1}{2}\tr(\eta_h)\bI$. Thus, \eqref{eq:infsup-low1} can be
written as:  
\begin{equation}\label{eq:infsup-low3D}
\div \eta_h = (\tilde{\tau}_{h,23}-\tilde{\tau}_{h,32},
\tilde{\tau}_{h,31}- \tilde{\tau}_{h,13},
\tilde{\tau}_{h,12}- \tilde{\tau}_{h,21})^{T}.
\end{equation}
Again, there exists a $\eta_h \in  \{\btau\in
H^1(\Omega;\mathbb{M})~|~\btau|_K
\in \cP_{k+2}(K;\mathbb{M})\}$ satisfying \eqref{eq:infsup-low3D} and 
\begin{equation*}
\|\rho_h\|_1 \lesssim \|\eta_h\|_1 \lesssim \|\tilde{\btau}_h\|_{0} \lesssim
\|v_h\|_0.
\end{equation*}
\end{enumerate}
To summarize, we obtain $\btau_h = \tilde{\btau}_h +\curl \rho_h $ that
satisfying $\btau_h \in \Sigma_{h,k+1}$, 
\begin{equation*}
\div\btau_h = v_h \quad \text{and} \quad \|\btau_h\|_{H(\div)}
\lesssim \|\tilde{\btau}_h\|_{H(\div)} + \|\curl \rho_h\|_0
\lesssim\|v_h\|_{0}.
\end{equation*} 
 This completes the proof.
\end{proof}

By virtue of Lemma \ref{lm:infsup-high} and \ref{lm:infsup-low}, we
have the following theorems.
   
\begin{theorem}\label{thm:wellposed}
Under the conditions in Lemma \ref{lm:infsup-high} or
\ref{lm:infsup-low}, the K-ellipticity \eqref{eq:K-ellipticity} and
the inf-sup condition \eqref{eq:infsup} hold uniformly with respect to
the mesh size. Consequently, the discrete mixed problem
\eqref{equ:discrete:formulation} is well posed.
\end{theorem}

\begin{theorem} \label{thm:error-estimate}
Let $(\bsigma, u)\in \Sigma \times V$ be the exact solution of the
problem \eqref{equ:mixed-formulation} and $(\bsigma_h, u_h)\in
\Sigma_{h,k+1} \times V_{h,k}$ the finite element solution of
\eqref{equ:discrete:formulation}.  Assume that $\bsigma \in
H^{k+2}(\Omega;\mS)$ and $u\in H^{k+1}(\Omega;\R^n)$. Under the
conditions in Lemma \ref{lm:infsup-high} or \ref{lm:infsup-low}, we
have  
\begin{equation}\label{eq:convergent}
\|\bsigma - \bsigma_h\|_{H(\div)} + \|u-u_h\|_{0} \lesssim
h^{k+1}(|\bsigma|_{k+2} + |u|_{k+1}).
\end{equation}
\end{theorem}
\begin{proof}
The well-posedness implies the following quasi-optimal error estimate,
\begin{equation}\label{eq:err_total}
\|\bsigma - \bsigma_h\|_{H(\div)} + \|u-u_h\|_{0} \lesssim
\inf_{\btau_h \in \Sigma_{h,k+1}, \bv_h \in V_{h,k} }( \| \bsigma -
    \btau_{h}\|_{H(\div)} + \|u - v_h\|_{0}),
\end{equation}
which gives rise to \eqref{eq:convergent} due to the standard  
$L^2$ projection and Scott-Zhang interpolation (cf.
\cite{scott1990finite}).
\end{proof}


\section{Hybridization}
\label{sec:hybrid}
To implement the mixed method \eqref{equ:discrete:formulation}, we
need the degrees of freedom (d.o.f.) or the nodal basis of the
discrete stress spaces. In the definition \eqref{eq:def-stress-space},
however, we state the inter-element continuity directly instead of
using the d.o.f., which is different from Ciarlet's convention for the
finite elements. More precisely, there is no locally defined d.o.f. on
elements for the discrete stress spaces \eqref{eq:def-stress-space}.
A similar argument can be found in \cite{arnold2002mixed}. In 
light of \cite{morgan1975nodal}, where the authors constructed the
nodal basis for the space of piecewise $C^1$ polynomials, we can
globally form the nodal basis for our discrete stress spaces,
whose dimensions depend on the singular vertices of the grids. 

Instead of presenting the details of the nodal basis, we adopt a
simpler implementation technique---the hybridization method (cf.
\cite{arnold1985mixed, cockburn2009unified}), which imposes the
inter-element continuity by Lagrange multiplier. The hybridization
method removes the inter-element continuity from the space
$\Sigma_{h,k+1}$, which results in a discontinuous stress space
\begin{equation} \label{eq:space-stress}
\Sigma_{h,k+1}^{-1} := \{ \btau_h \in L^2(\Omega; \mS)~|~ \btau_h |_{K}
\in \cP_{k+1}(K; \mS)\quad \forall K \in \cT_h\}.
\end{equation}
To enforce the inter-element continuity of the stress, we introduce the
Lagrange multiplier space $M_{h,k+1}$, where 
\begin{equation} \label{eq:multiplier-space}
M_{h, k+1} := \{ \mu_h \in L^2(\cF_h, \R^n)~|~ \mu_h |_{F} \in
\cP_{k+1}(F,\R^n) \quad \forall F \in \cF_h^i, \mbox{ and } \mu_h
|_{\cF^{\partial}_h} = 0 \}. 
\end{equation}

The hybridized mixed finite element method is to find $(\bsigma_h,
u_h, \lambda_h )\in \Sigma^{-1}_{h,k+1} \times V_{h,k} \times
M_{h,k+1} $ satisfying
\begin{subequations}\label{eq:hybrid-system}
\begin{align}
(\mathcal{A} \bsigma_h , \btau_h)_{\cT_h} &+ ( \div_h \btau_h,
u_h)_{\cT_h} - \langle [ \btau_h ], \lambda_h\rangle _{\cF_h^i} = 0
~\quad \qquad \forall \btau_h \in \Sigma^{-1}_{h,k+1},
\label{eq:hybrid-systemA}\\
(\div_h \bsigma_h, v_h )_{\cT_h} & \qquad \qquad \qquad \qquad \qquad
\qquad \qquad = (f, v_h) \quad \forall v_h \in V_{h,k},
\label{eq:hybrid-systemB}\\
-\langle [\bsigma_h], \mu_h \rangle _{\cF_h^i} & \qquad \qquad \qquad
\qquad \qquad \qquad \qquad = 0 ~\quad \qquad \forall \mu_h \in
M_{h,k+1}.
 \label{eq:hybrid-systemC}
\end{align}
\end{subequations}
Here, $\div_h$ is the broken divergence operator. For convenience, let
$\mathcal{B} := \div_h: \Sigma_{h,k+1}^{-1} \mapsto V_{h,k}$ and
$\mathcal{C}:\Sigma_{h,k+1}^{-1} \mapsto M_{h,k+1}$ defined by
\begin{equation} \label{eq:operator-C}
\mathcal{C} \btau|_{F} := 
\begin{cases}
[\btau]|_F & \mbox{for } F \in  \cF^i_{h}, \\ 
0 & \mbox{for } F \in \cF^{\partial}_{h}.
\end{cases}
\end{equation}
The adjoint of these operators are defined as $\mathcal{B}^*: V_{h,k}
\mapsto \Sigma_{h,k+1}^{-1}$ and $\mathcal{C}^*:M_{h,k+1}
\mapsto\Sigma_{h,k+1}^{-1}$ such that for any $(\btau_h, v_h, \mu_h) \in
\Sigma_{h,k+1}^{-1} \times V_{h,k} \times M_{h,k+1}$,
\begin{equation*}
(\mathcal{B}^* v_h, \btau_h) = (v_h,
\mathcal{B}\btau_h) \quad\text{and}\quad  (\mathcal{C}^* \mu_h, \btau_h)
= \langle\mu_h, \mathcal{C} \btau_h \rangle_{\cF_h^i}.
\end{equation*}

The following theorem shows the property of hybridized method
given in \eqref{eq:hybrid-system}. 

\begin{theorem}\label{thm:hybrid-well-posed}
There exists a solution $(\bsigma_h, u_h, \lambda_h )\in
\Sigma^{-1}_{h,k+1} \times V_{h,k} \times M_{h,k+1} $ for the
hybridized system \eqref{eq:hybrid-system}. Moreover, the first two
components of the solution are unique and coincide with that of the
mixed method \eqref{equ:discrete:formulation}.
\end{theorem}
\begin{proof}
By Theorem \eqref{thm:wellposed}, there exists a solution $(\bsigma_h,
u_h) \in \Sigma_{h,k+1} \times V_{h,k}$ for the mixed method
\eqref{equ:discrete:formulation}.  It is obvious that $(\bsigma_h,
u_h)$ satisfies the last two equations \eqref{eq:hybrid-systemB}
and \eqref{eq:hybrid-systemC}. The first equation
\eqref{eq:hybrid-systemA} can be written as 
\begin{equation} \label{eq:operatorForm}
\mathcal{C}^* \lambda_h =   \mathcal{A} \bsigma_h + \mathcal{B}^* u_h.
\end{equation}
Since ${\rm R}(\mathcal{C}^*)^{\bot} = \ker(\mathcal{C})$ and
$\ker(\mathcal{C}) = \Sigma_{h,k+1}$ , we have 
$$  
{\rm R}(\mathcal{C}^*)= (\Sigma_{h,k+1})^\bot. 
$$
Here, $(\Sigma_{h,k+1})^\bot$ is the $L^2$ orthogonal complement of
$\Sigma_{h,k+1}$ in  $\Sigma_{h,k+1}^{-1}$ with respect to the inner
product $(\cdot, \cdot)$. Since $(\bsigma_h, u_h)$ satisfies
\eqref{eq:hybrid-systemA} for $\btau_h \in \Sigma_{h,k+1}$, it holds
that 
\begin{equation} \label{eq:mixed-existence}
\mathcal{A} \bsigma_h + \mathcal{B}^* u_h \in (\Sigma_{h,k+1})^\bot.
\end{equation}
Hence, there exists $ \lambda_h \in M_{h,k+1}$ satisfying
\eqref{eq:operatorForm}, which indicates the existence of the solution
for \eqref{eq:hybrid-system}.

For the uniqueness, assuming that $(\bsigma_h, u_h, \lambda_h )$
satisfies \eqref{eq:hybrid-system}, then \eqref{eq:hybrid-systemC}
implies that $\bsigma_h\in \Sigma_{h,k+1}$. Moreover, since
$\Sigma_{h,k+1}\subset \Sigma_{h,k+1}^{-1}$, choosing $\btau_h \in
\Sigma_{h,k+1}$, we can see that the system \eqref{eq:hybrid-systemA}
and \eqref{eq:hybrid-systemB} is identical to the system of the mixed
method \eqref{equ:discrete:formulation}. Therefore, $(\bsigma_h, u_h)$
solves \eqref{equ:discrete:formulation}. The uniqueness of
$(\bsigma_h, u_h)$ follows from Theorem \ref{thm:wellposed}. This
completes the proof.
\end{proof}

\begin{remark}
We note that Hu-Zhang elements in \eqref{eq:Hu-Zhang} can also be
written as 
$$
\begin{aligned}
\Sigma_{h,k+1}^{\rm HZ} = \{\btau  \in H(\div,\Omega;\mS)~|~&
\btau|_K \in \cP_{k+1}(K;\mS)\quad \forall K\in \cT_h,\\
 &\text{and } \btau|_a ~\text{is continuous for any } a \in \cN_h\}.
 \end{aligned}
$$ 
We enrich the space $\Sigma_{h,k+1}^{\rm HZ}$ by relaxing the
continuity on the element vertices. Similar technique can be used in
Arnold-Winther \cite{arnold2002mixed} ($n=2$) or 
Arnold-Awanou-Winther \cite{arnold2008finite} ($n=3$) elements
$\tilde{\Sigma}_{h,k+n-}-V_{h,k}$, where
$$
\begin{aligned}
\tilde{\Sigma}_{h,k+n-} := \{\btau  \in H(\div,\Omega;\mS)~|~&
\btau|_K \in \tilde{\Sigma}_{k+n-}(K) \quad \forall K\in \cT_h, \\
&\text{and }\btau|_a ~\text{is continuous for any } a \in \cN_h\}.
 \end{aligned}
$$ 
and $\tilde{\Sigma}_{k+n-}(K) := \{\btau \in \cP_{k+n}(K;\mS)~|~\div
\btau \in \cP_k(K;\R^n)\}$. We denote the hybridized version of
$\tilde{\Sigma}_{h,k+n-}$ by 
$$
\Sigma_{h,k+n-} := \{\btau  \in H(\div,\Omega;\mS)~|~
\btau|_K \in \tilde{\Sigma}_{k+n-}(K) \quad \forall K\in \cT_h\}.
$$
Table \ref{tb:HZvsAW} compares $\tilde{\Sigma}_{h,k+n-}$ and
$\Sigma_{h,k+1}^{\rm HZ}$ to their hybridized versions. 

\begin{table}[!htbp]
\caption{$\tilde{\Sigma}_{h,k+n-}$, $\Sigma_{h,k+1}^{\rm HZ}$ and
their hybridized versions.}\centering
{\small{
\begin{tabular}{|c|c|c|c|c@{}|}
  \hline
\multicolumn{1}{|c|}{Elements}	&   \multicolumn{1}{|c|}{Gerneral
  grids}	&  \multicolumn{1}{|c|}{Special grids}	&
  \multicolumn{1}{|c|}{Hybridizable} & \multicolumn{1}{|c|}{Lagrange
multiplier}\\
  \hline
 $\tilde{\Sigma}_{h,k+n-}-V_{h,k}$	& $k\ge1$	& -- 		& $\times$	& --\\ 
 $\Sigma_{h,k+1}^{\rm HZ}-V_{h,k}$	& $k\ge n$	& -- 		& $\times$	& --\\ 
 $\Sigma_{h,k+n-}-V_{h,k}$		& $k\ge1$	& -- 		& $\surd$		& $M_{h,k+n}$\\ 
 $\Sigma_{h,k+1}-V_{h,k}$		& $k\ge n$	& $k\ge n-2$ & $\surd$	& $M_{h,k+1}$ 
   \\ \hline
\end{tabular} 
}}
\label{tb:HZvsAW}
\end{table}
\end{remark}

\subsection{Kernel of the Hybrid System}

Theorem \ref{thm:hybrid-well-posed} implies that the kernel of the
hybridized mixed system \eqref{eq:hybrid-system} is  $\{0\} \times
\{0\} \times \ker(\mathcal{C}^*)$. It is straightforward that  
$
\ker(\mathcal{C}^*)  = {\rm R}(\mathcal{C})^{\bot},
$
where ${\rm R}(\mathcal{C})^{\bot}$ is the $L^2$ orthogonal complement
of ${\rm R}(\mathcal{C})$ in the space $M_{h,k+1}$ with respect to the
inner product $\langle\cdot, \cdot\rangle_{\cF_h^i}$.  We therefore
have the following decomposition for the multiplier space 
$$
M_{h,k+1} = {\rm R}(\mathcal{C})\oplus {\rm R}(\mathcal{C})^{\bot}.
$$
We note that the dimension of ${\rm R}(\mathcal{C})^{\bot}$ depends on
the grid.

\begin{definition}[\cite{morgan1975nodal}]
In the 2D case, an interior vertex $a\in \cN_h$ ($a\not \in\partial
\Omega$) is called singular,  if and  only if the  edges meeting
at this vertex   fall on two straight lines. 
\end{definition}
\begin{figure}[h!]
\centering\includegraphics[scale =.5]{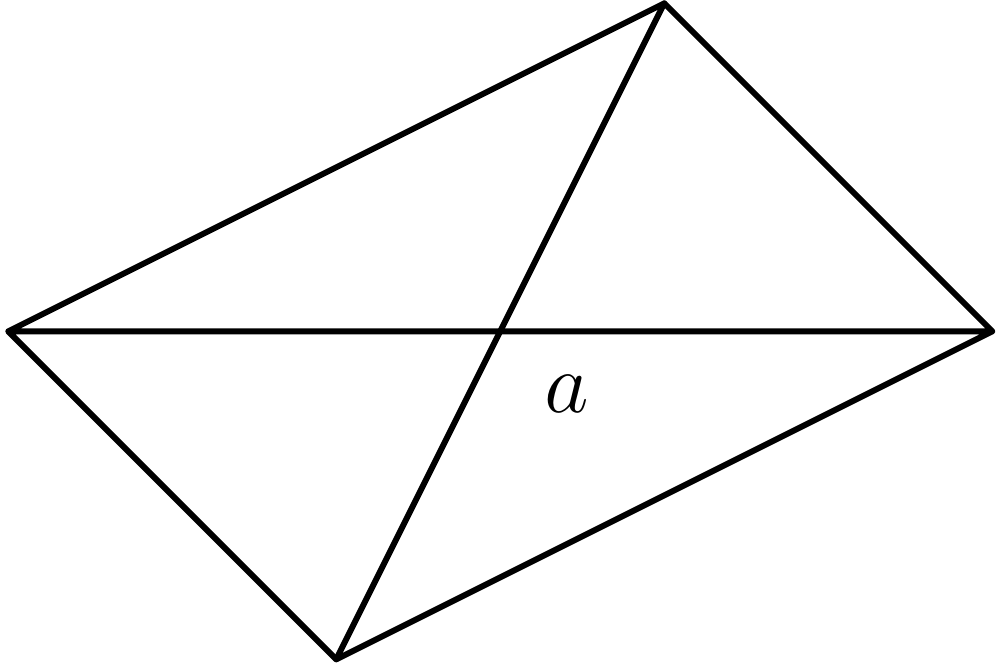}
\caption{Singular vertex $a$.}
\label{fig:singular-vertex}
\end{figure}

\begin{lemma}\label{lm:local-basis}
Both ${\rm R}(\mathcal{C})$ and ${\rm R}(\mathcal{C})^{\bot}$ have
local basis, that is,
$$
{\rm R}(\mathcal{C}) = \mathrm{span}\{ \varphi_1,
  \varphi_2,\cdots,\varphi_{N_1}\},
$$
$$
{\rm R}(\mathcal{C})^{\bot} = \mathrm{span}\{\psi_1,
\psi_2,\cdots,\psi_{N_2}\},
$$
where $\varphi_i, \psi_j \in M_{h,k+1}$ are locally supported and
$N_1, N_2$ are the dimensions of the spaces ${\rm R}(\mathcal{C}),
{\rm R}(\mathcal{C})^{\bot}$, respectively. Moreover, for the 2D case, if
there is no interior singular vertex in $\cT_h$, we have $N_2 =
0$ and $\{\varphi_1, \cdots, \varphi_{N_1}\}$ can be chosen such that  
the mass matrix $\boldsymbol{M} = (\langle \varphi_i,
\varphi_j\rangle_{\cF_h^i})$ is well-conditioned, that is,
\begin{equation} \label{eq:basis-L2-stability}
\sum_{i=1}^{N_1} c_i^2 \|\varphi_i\|_0^2 \eqsim \|\sum_{i=1}^{N_1} c_i
\varphi_i\|_0^2 \qquad \forall (c_1, c_2, \cdots, c_{N_1}) \in
\mathbb{R}^{N_1}. 
\end{equation}
\end{lemma}
\begin{proof}
The detailed proof is given in the appendix. 
\end{proof}

\subsection{SPSD System for Lagrange Multiplier}

In this subsection, we eliminate the variable $\bsigma_h$ and $u_h$ in
the hybridized mixed system \eqref{eq:hybrid-system}, then obtain a
linear system solely for $\lambda_h$.

For any $\lambda \in M_{h,k+1}$, we define two local problems: 

\begin{enumerate}
\item Find $(\bsigma_{\lambda}, u_{\lambda}) \in
\Sigma_{h,k+1}^{-1} \times V_{h,k}$ such that for any element $K\in
\cT_h$, 
\begin{subequations}\label{eq:lsm}
\begin{align}
&(\mathcal{A} \bsigma_ {\lambda}, \btau_h )_{K} + ( u_{\lambda}, \div
\btau_h)_{K} = \langle \lambda, \btau_h \nu\rangle_{\partial
  K} \qquad \forall \btau_h  \in \cP_{k+1}(K;\mS),\label{eq:lsm1}\\
&(\div \bsigma_\lambda, v_h)_{K} ~ \qquad \qquad \qquad = 0 ~ \qquad
\qquad \qquad \forall v_h \in \cP_{k}(K;\mathbb{R}^n).\label{eq:lsm2}
\end{align}
\end{subequations}
\item Find $(\tilde{ \bsigma}_{f} , \tilde{u}_{f}) \in
\Sigma_{h,k+1}^{-1} \times V_{h,k}$ such that for any element $K\in
\cT_h$,
\begin{subequations}\label{eq:lsf}
\begin{align}
&(\mathcal{A} \tilde{\bsigma}_{ f}, \btau_h)_{K} + ( \tilde{u}_{
f}, \div \btau_h)_{K} =0 \quad \qquad \qquad \forall \btau_h  \in
\cP_{k+1}(K;\mS), \label{eq:lsf1} \\
&(\div \tilde{\bsigma}_{f},  v_h)_{K} ~ \qquad \qquad \qquad = (f, v_h
)_{K} \qquad \forall v_h \in \cP_{k}(K;\mathbb{R}^n). \label{eq:lsf2}
\end{align}
\end{subequations}
\end{enumerate} 

The following lemma shows that both $(\bsigma_m, u_m)$ and $(\tilde{
\bsigma}_{f}, \tilde{ u}_{f})$  are well defined. 

\begin{lemma}\label{lm:loc_recover}
The systems \eqref{eq:lsm}  and \eqref{eq:lsf} are unisolvent.
Moreover, the solution of the system \eqref{eq:hybrid-system}
satisfies 
\begin{equation}\label{eq:loc_recover}
\bsigma_h =\bsigma_{ \lambda_h }+\tilde{\bsigma}_{ f}  \text{\qquad
  and     \qquad}  u_h = u_{\lambda_h} +\tilde{u}_{ f}. 
\end{equation}
\end{lemma}
\begin{proof}
The proof is similar to the standard one given in
\cite{cockburn2009unified} and is therefore omitted here. 
\end{proof}

Note that $(\bsigma_{\lambda_h}, u_{\lambda_h})$ and $(\tilde{
\bsigma}_{f} , \tilde{ u}_{f})$ can be computed element by
element. The above lemma means that the $\bsigma_h$ and $u_h$ can be
locally recovered after solving the variable $\lambda_h$. 

\begin{theorem} \label{tm:hybrid-multiplier}
The Lagrange multiplier $\lambda_h$  satisfies 
\begin{equation}\label{eq:multiplier}
s(\lambda_h, \mu_h) =-(f, u_{ \mu_h})  \qquad \forall \mu \in M_{h,k+1},
\end{equation}
where $s(\lambda_h, \mu_h) =(\mathcal{A} \bsigma_{\lambda_h}, \bsigma_{
\mu_h})$. Moreover, the system \eqref{eq:multiplier} is symmetric
positive-semidefinite and its kernel is ${\rm R}(\mathcal{C})^{\bot}$.
\end{theorem}
\begin{proof}
The derivation of \eqref{eq:multiplier} is standard in the
hybridization method (cf. \cite{cockburn2009unified}). The kernel of
the multiplier system is the same with the hybridized mixed system.
\end{proof}

\subsection{Norm Estimates}
We denote the linear operator corresponding to the bilinear form
$s(\cdot, \cdot)$ by $\mathcal{S}:M_{h,k+1}\mapsto M_{h,k+1}$, or $S:
M_{h,k+1} \mapsto M_{h,k+1}'$ as 
\begin{equation} \label{eq:operator-S} 
\langle S\lambda, \mu \rangle := \langle \mathcal{S}\lambda,
\mu \rangle_{\cF_h^i} := s(\lambda, \mu) \qquad \forall \lambda,
\mu \in M_{h,k+1}.
\end{equation} 
In fact, $\mathcal{S}$ is the Schur complement of the hybridized mixed
system \eqref{eq:hybrid-system}.  In light of Theorem
\ref{tm:hybrid-multiplier}, we can define a norm $\|\cdot\|_{S}$ on
$\mathrm{R}(\mathcal{C})$ as 
\begin{equation} \label{eq:norm-S}
\| \lambda \|_{S}^2 :=  \sum_{K\in\cT_h} \|\lambda\|_{S,K}^2  :=
\sum_{K\in\cT_h} (\mathcal{A} \bsigma_{\lambda},
\bsigma_{\lambda})_K \qquad \forall \lambda\in
\mathrm{R}(\mathcal{C}),
\end{equation}
which can also be extended as a semi-norm on $M_{h,k+1}$. For the
conciseness, we still denote the semi-norm on $M_{h,k+1}$ by
$\|\cdot\|_S$. 

To investigate how $\|\cdot\|_S$ depends on the parameters, we define
the following semi-norms locally:
\begin{align}
|\lambda|_{h, K} &:= \sup_{\btau \in Z_h(K)} \frac{\langle \lambda,
\btau \nu \rangle_{\partial K}}{\|\btau\|_{0,K}} \quad \qquad
\forall \lambda \in M_{h,k+1}, \label{eq:norm-h}\\
|\lambda|_{*, K} &:=|K|^{-1/2}\left| \int_{\partial K}  \lambda
\cdot \nu ~ds \right| \qquad \forall \lambda \in M_{h,k+1}.
\label{eq:norm-star}
\end{align}
Here, $Z_h(K) = \{\btau_h \in \mathcal{P}_{k+1}(K;\mS)~|~ \div \btau_h =
0\}$. The semi-norms  $|\cdot|_{h}$ and  $|\cdot|_{*}$ on $M_{h,k+1}$
are defined by the summations of local norms over all elements,
namely, 
$$
|\lambda|_{h}^2 = \sum_{K\in \cT_h} |\lambda|_{h, K}^2
\quad\text{and}\quad 
|\lambda|_{*}^2 = \sum_{K\in \cT_h} |\lambda|_{*, K}^2.
$$
The relationship between $\|\cdot\|_{S}$ and $|\cdot|_{*}$, $| \cdot |_{h}$ is
described in the following  lemma. 
\begin{theorem}\label{lm:equivalent-norm}
It holds that
\begin{equation} \label{eq:equivalent-energy-norm}
\| \lambda \|^2_{S, K} \eqsim 2\tilde{\mu} |\lambda|_{h,K}^2 +
\tilde{\lambda} | \lambda |_{*, K}^2    \qquad \forall \lambda \in
M_{h,k+1}.
\end{equation}
\end{theorem}
\begin{proof}
Notice that $\bsigma_{{\lambda}}|_{K} \in Z_h(K)$ by \eqref{eq:lsm2}.
Moreover, for any $\btau \in Z_h(K)$, by \eqref{eq:lsm1}, we have
\begin{equation}\label{eq:lam2sig}
(\mathcal{A} \bsigma_{{\lambda}}, \btau)_{K} = \langle
{\lambda}, \btau\nu \rangle_{\partial K}.
\end{equation} 
Let $m_{\btau} = \frac{1}{n|K|} \int_K \tr(\btau) dx$ and $\btau_0 =
\btau - m_{\btau}\bI$. Then $(\btau_0, \bI)_K=0$ and $(\mathcal{A} \btau_0,
\bI)_K=0$, which implies that 
$$   
(\mathcal{A}\btau, \btau)_K = (\mathcal{A}\btau_0, \btau_0)_K +
(\mathcal{A}m_{\btau} \bI, m_{\btau} \bI)_K. 
$$
Let $\|\btau\|_{\mathcal{A},K}:= (\mathcal{A}\btau, \btau)_K^{1/2}$
for any $\btau \in L^2(K;\mS)$.  In light of \eqref{eq:lam2sig} and
\eqref{eq:incompressible}, we have for any $\lambda \in M_{h,k+1}$,
$$
\begin{aligned}
\|{\lambda}\|_{S, K} &= \sup_{\btau \in Z_h(K)}
\frac{(\mathcal{A}\bsigma_{\lambda},
\btau)_K}{\|\btau\|_{\mathcal{A},K}}  
= \sup_{\btau \in Z_h(K)} \frac{\langle \lambda, \btau
  \nu\rangle_{\partial K}}{\|\btau\|_{\mathcal{A},K}} \\  
&\le \sup_{\btau \in Z_h(K)} \frac{\langle \lambda, \btau_0
  \nu\rangle_{\partial K}}{\|\btau\|_{\mathcal{A},K}}
  + \sup_{\btau \in Z_h(K)} \frac{\langle \lambda, m_{\btau} \bI
    \nu\rangle_{\partial K}}{\|\btau\|_{\mathcal{A},K}} \\  
&= \sup_{\btau \in Z_h(K)} \frac{\langle \lambda, {\btau}_0
  \nu\rangle_{\partial K}}{\|\btau_0\|_{\mathcal{A},K}} +
  \sup_{\btau \in Z_h(K)} \frac{\langle \lambda, m_{\btau} \bI
  \nu\rangle_{\partial K}}{\|m_{\btau} \bI\|_{\mathcal{A},K}} \\  
&\lesssim ( 2\tilde{\mu})^{1/2} \sup_{\btau \in Z_h(K)} \frac{\langle
  \lambda, {\btau}_0 \nu\rangle_{\partial K}}{\| {\btau}_0 \|_{0,K}} +
  \tilde{\lambda }^{1/2}|\lambda|_{*,K} \\  
&\lesssim (2\tilde{\mu})^{1/2} |\lambda|_{h,K} + \tilde{\lambda
}^{1/2}|\lambda|_{*,K}.
\end{aligned}
$$

On the other hand, since $2\tilde{\mu} (\mathcal{A}\btau, \btau)_K
\lesssim \|\btau\|_{0, K}^2$ by the definition of $\mathcal{A}$, we
have
$$
 (2\tilde{\mu})^{1/2}|\lambda|_{h,K} = (2\tilde{\mu})^{1/2}\sup_{\btau
 \in Z_h(K)} \frac{\langle \lambda, {\btau} \nu\rangle_{\partial
 K}}{\| {\btau} \|_{0,K}} \lesssim  \sup_{\btau \in Z_h(K)}
 \frac{\langle \lambda, \btau \nu\rangle_{\partial
 K}}{\|\btau\|_{\mathcal{A},K}} = \|{\lambda}\|_{S, K}.
$$
Moreover, we have $(\mathcal{A} \bsigma_{{\lambda}},\bI )_{K} = \langle
{\lambda}, \bI\nu \rangle_{\partial K}$ from \eqref{eq:lsm1}. By
the Cauchy-Schwarz inequality, we have 
$$
\begin{aligned}
\tilde{\lambda }^{1/2}|\lambda|_{*,K} & =
\tilde{\lambda}^{1/2}|K|^{-1/2}|\langle \lambda, \bI
\nu\rangle_{\partial K}| 
= \tilde{\lambda}^{1/2}|K|^{-1/2}|(\mathcal{A}\bsigma_{\lambda},
\bI)_K| \\
& \leq \tilde{\lambda}^{1/2}|K|^{-1/2}\|\lambda\|_{S,K}
(\mathcal{A}\bI, \bI)_K^{1/2} \leq \|{\lambda}\|_{S, K}.
\end{aligned}
$$
This completes the proof.
\end{proof}

Next, we estimate the condition number of $S$. The $L^2$ norm for
$M_{h,k+1}$ is denoted by 
$$ 
\|\lambda\|_{0}^2 := \sum_{F\in \cF_h^i} \|\lambda\|_{0, F}^2
:=\sum_{F\in \cF_h^i} \langle \lambda, \lambda \rangle_{F}. 
$$

\begin{lemma}\label{lm:L2-upper}
It holds that 
\begin{equation} \label{eq:L2-upper}
 \| \lambda \|_{S}^2 \lesssim  (2\tilde{\mu}+ \tilde{\lambda})
  h^{-1}\|\lambda\|_0^2 \qquad \forall \lambda \in M_{h,k+1}.
\end{equation} 
\end{lemma}
\begin{proof}
The upper bound of $S$ follows from the equivalent norm
\eqref{eq:equivalent-energy-norm}, Cauchy-Schwarz inequality and
standard scaling argument.
\end{proof}

The lower bound of $S$ depends on the singularity of the grids. In
light of \cite{scott1985norm}, we define a quantity to measure the
vertex singularity. The rest estimates are focused on the case of 
spatial dimension $n=2$. For a vertex $a \in \cN_h$, let $\theta_i,
1\le i\le m$ be the angles of the triangle $K_i$ meeting at $a$ (triangles are numbered consecutively). If $a$ is an internal
vertex, we define
$$
\kappa(a) := \max \{ |\theta_i + \theta_j - \pi|~\big|~1\le i, j\le
m~\text{and}~ {i-j=1\mod m} \};
$$
If $a$ is a boundary vertex, $\kappa(a)$ is defined in the same way
without the modulo operation.
We further set
$$
\kappa = \min_{a \in \cN_h } \kappa(a).
$$

\begin{figure}[htbp!]
\centering\includegraphics[scale =.7]{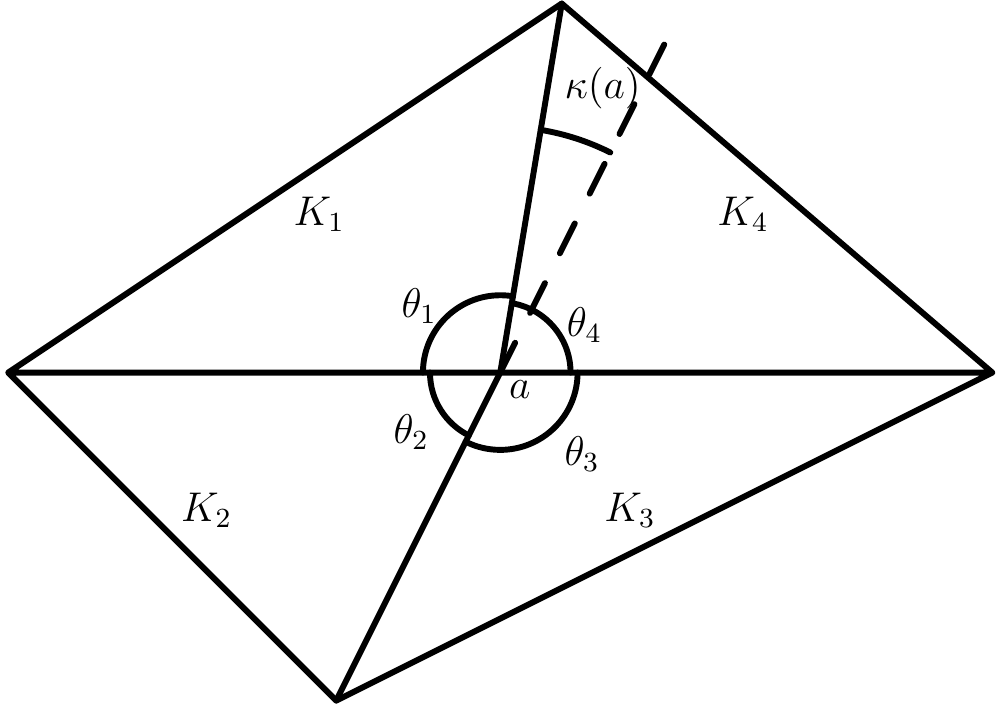}
\caption{Nearly singular vertex.}
\label{fig:nearly-singular-vertex}
\end{figure}
In the following, we assume that $\kappa\ge \kappa_0 >0$, where
$\kappa_0$ is a positive constant independent of $h$. That is, there
is no singular or nearly singular vertex on $\cT_h$. 


\begin{lemma}\label{lm:stability-local-jump}
Assume that $\kappa\ge \kappa_0 >0$. For any local basis function
$\varphi_i$ of $\mathrm{R}(\mathcal{C})$ (see Lemma
\ref{lm:local-basis}), there exists a locally supported $\btau_i\in
\Sigma_{h,k+1}^{-1}$ such that 
$$
[\btau_i]|_{F} = \varphi_i|_F \quad\forall F\in \cF_h, \quad \text{and}
\quad \|\btau_i\|_0^2 \lesssim \frac{h}{ \sin^2(\kappa_0)} \|\varphi_i
\|_{0}^2.
$$
\end{lemma}
\begin{proof}
The detail proof for above lemma will be given in the appendix.
\end{proof}

\begin{lemma}\label{lm:stableJ}
Assume that $\kappa\ge \kappa_0 >0$. For any $\lambda \in M_{h,k+1}$, there
exists $\btau \in \Sigma_{h,k+1}^{-1}$ such that 
\begin{equation} \label{eq:stableJ}
 [\btau]|_F = {\lambda}|_F \quad \forall F\in \cF_h, \quad \text{and}
 \quad \|\btau\|_{0}^2 \lesssim \frac{h}{\sin^{2}(\kappa_0)}
 \|{\lambda}\|_{0}^2. 
\end{equation}
\end{lemma}
\begin{proof}
Since $\kappa\geq \kappa_0>0$, by Lemma \ref{lm:local-basis}, we have
$M_{h,k+1} = \mathrm{R}(\mathcal{C})$ and there exists a local basis
that satisfies \eqref{eq:basis-L2-stability}. Therefore, any $\lambda \in
M_{h,k+1}$ can be uniquely expressed as 
$$
\lambda = \sum_{i=1}^{N_1} c_i \varphi_i \quad \text{and} \quad
\sum_{i=1}^{N_1} \|c_i \varphi_i\|_{0}^2 \eqsim \|\lambda\|_{0}^2.
$$
By virtue of Lemma \ref{lm:stability-local-jump}, there exists a
locally supported $ \btau_i \in \Sigma_{h,k+1}^{-1} $ for each basis
function $\varphi_i$ of $\mathrm{R}(\mathcal{C})$, such that 
$$
[\btau_i]|_F= \varphi_i|_F  \quad \forall F\in \cF_h, \quad \text{and}
\quad \|\btau_i\|_0^2 \lesssim h\sin^{-2}(\kappa_0)
\|\varphi_i\|_{0}^2. 
$$
Therefore, $\btau = \sum_{i=1}^{N_1} c_i{\btau_i}$ satisfies
$[\btau]|_F = \lambda$ and 
$$ 
\|\btau\|_0^2 \lesssim \sum_{i=1}^{N_1} c_i^2 \|\btau_i\|_0^2 \lesssim
h\sin^{-2}(\kappa_0) \sum_{i=1}^{N_1} c_i^2 \|\varphi_i\|_0^2 \eqsim
h\sin^{-2}(\kappa_0) \|\lambda\|_0^2. 
$$ 
This completes the proof.
\end{proof}

\begin{lemma}\label{lm:L2-lower}
Assume that $\kappa\ge \kappa_0 >0$. It holds that 
\begin{equation} \label{eq:L2-lower}
2\tilde{\mu} h \sin^{2}(\kappa_0)\|\lambda\|_0^2  ~ \lesssim
~\|\lambda \|_{S}^2 \qquad \forall \lambda \in M_{h,k+1}.
\end{equation}
\end{lemma}
\begin{proof} 
By virtue of Lemma \ref{lm:stableJ}, for any $\lambda \in
M_{h,k+1}$, there exists $\btau_1 \in \Sigma_{h,k+1}^{-1}$ such
that 
$$ 
\mathcal{C} \btau_1 = {\lambda} \quad \text{and} \quad \|\btau_1\|_{0}
\lesssim h^{1/2} \sin^{-1}(\kappa_0)\|\lambda\|_{0}.
$$ 
Applying the discrete inf-sup condition, there exists $\btau_2 \in
\Sigma_{h,k+1}$ such that 
$$ 
\div\btau_2 = -\div_h \btau_1 \quad \text{and} \quad
\|\btau_2\|_{H(\div)} \lesssim \|\div_h \btau_1\|_0 \lesssim h^{-1}
\|\btau_1\|_0 \lesssim h^{-1/2}\sin^{-1}(\kappa_0)\|{\lambda}\|_{0}. 
$$
Let $\btau = \btau_1 + \btau_2$. Thus, $\div_h\btau = 0$ and
$\mathcal{C}\btau = {\lambda}$.  By summation of \eqref{eq:lsm1} over
all elements and choosing above $\btau$ as a testing function, we have
$$ 
\|\lambda\|_{0}^2 = (\mathcal{A} \bsigma_{\lambda}, \btau) \leq
(\mathcal{A} \bsigma_{\lambda},
\bsigma_{\lambda})^{1/2} (2\tilde{\mu})^{-1/2} \|\btau\|_{0}  \lesssim
(2\tilde{\mu}h)^{-1/2}\sin^{-1}(\kappa_0) \|\lambda\|_{S}
\|\lambda\|_{0},
$$ 
which implies \eqref{eq:L2-lower}.
\end{proof}

Lemmas \ref{lm:L2-upper} and \ref{lm:L2-lower} imply the following
condition number estimate:
\begin{equation} \label{eq:cond-S}
\mathrm{cond}(S) \lesssim \frac{2\tilde{\mu}+
  \tilde{\lambda}}{2\tilde{\mu}} h^{-2}\sin^{-2}(\kappa_0).
\end{equation}
For the nearly incompressible material, $\tilde{\lambda}$ would be
sufficient large, which makes the multiplier system
\eqref{eq:multiplier} nearly singular.


\section{Multilevel Solvers for the Hybridized Mixed Methods} \label{sec:schwarz}
In this section, we shall describe several multilevel solvers for the
hybridized mixed methods for the 2D case.  We further assume that
$\kappa\ge \kappa_0>0$, which is guaranteed when the grid has no
singular or nearly singular vertex.

\subsection{Two-level and multilevel solvers}
First, we present the two-level solvers.  We consider an overlapping
decomposition $\{\Omega_i\}_{i=1}^J$, where $\Omega_i$ are open
subdomains of $\Omega$.  Let $\cT_H$ be a coarse grid for $\Omega$,
and $\cT_h$ be a subdivision of $\cT_H$ such that $\cT_h$ is aligned
with each $\partial \Omega_i$.  We assume that there exist nonnegative
$C^\infty$ functions $\theta_1, \theta_2, \cdots, \theta_J$ in $\R^2$
such that 
\begin{subequations}
\begin{align}
& \theta_i = 0 \qquad \text{on}~\Omega\setminus\Omega_i,
  \label{equ:unity-support} \\ 
& \sum_{i=1}^J \theta_i = 1 \qquad \text{on}~\bar{\Omega}, 
  \label{equ:unity-sum} \\ 
& \|\nabla \theta_i\|_{\infty} \lesssim \delta^{-1}. 
  \label{equ:unity-derivative}
\end{align}
\end{subequations}
Here, $\delta > 0$ is a parameter that measures the overlap
among the subdomains. We also assume that there exists an integer
$N_c$ independent of $h$, $\delta$, and $J$ such that any point in
$\Omega$ belongs to at most $N_c$ subdomains. 
The local space associated with subdomain $\Omega_i$ is denoted by
\begin{equation} \label{eq:local-space-M}
M_i := \{\lambda \in M_{h,k+1} ~|~ \lambda|_{F} = 0, \text{ for any 
face } F\in \Omega\backslash \Omega_i\}.
\end{equation} 
We can then define $S_i: M_i \mapsto M_i'$ and bilinear form on $M_i$ by 
$$ 
\langle S_i \lambda_i, \mu_i \rangle := s_i(\lambda_i, \mu_i) :=
s(\iota_i \lambda_i, \iota_i \mu_i),
$$ 
where $\iota_i:M_i \hookrightarrow M_{h,k+1}$ denotes the inclusion
operator.

In light of the multigrid method on the primal elasticity problem by
Sch{\"o}berl \cite{schoberl1999multigrid}, we choose the continuous
and piecewise quadratic finite element space as the coarse space
$$
W_H : = \{ w \in H_0^1(\Omega; \R^2)~|~ w|_K \in \cP_2(K;\R^2) \quad
\forall  K\in \cT_H \}.
$$
Suppose the coarse space $W_H$ is connected to $M_{h,k+1}$ by an
injective intergrid operator $I_H^h: W_H \mapsto M_{h,k+1}$ (The
construction of $I_H^h$ will be given in Section
\ref{subsec:intergrid}). We will impose the primal elastic norm  
$\|\cdot\|_{A_H}$ on $W_H$, where $A_H: W_H \mapsto W_H'$ and the
bilinear form $a_H(\cdot, \cdot)$ are defined as
\begin{equation} \label{eq:elastic-norm}
\begin{aligned}
\langle A_H w_H, v_H \rangle & := a_H(w_H, v_H) \\ 
& :=  2\tilde{\mu} ( \bepsilon(w_H), \bepsilon(v_H) ) + \tilde{\lambda}(
    P_0^H \div w_H, P_0^H\div v_H ) \qquad \forall w_H, v_H \in W_H, \\
\|w_H\|_{A_H}^2 &:= a_H(w_H, w_H) \qquad \forall w_H \in W_H.
\end{aligned}
\end{equation}
Here, $P_0^H$ is the $L^2$ projection on the piecewise constant space
on $\cT_H$. 
Then, the two-level additive Schwarz preconditioner can be
constructed as 
\begin{equation} \label{eq:additive-precond2}
B_{\rm ad} = I_H^h A_H^{-1} (I_H^h)' + \sum_{i=1}^J \iota_i S_i^{-1}
\iota_i'. 
\end{equation}
Our main contribution is the following estimate.
\begin{theorem}\label{thm:additive-condition-number}
The condition number of $B_{\rm ad}S$ satisfies 
$$
\mathrm{cond}(B_{\rm ad}S) \le C (1+N_c) \frac{H^2}{\delta^2},
$$
where $C$ is independent to both the mesh size $h$ and the Lam\'e
constants. 
\end{theorem}
According to \cite{xu1992iterative, toselli2005domain}, the estimate
of the condition number of the additive Schwarz method is based on
the stability of the intergrid transfer operator $I_H^h$ (Lemma
\ref{lm:I-stability}) and the stable decomposition (Theorem
\ref{lm:stable-decomposition}), which will be proved in the rest of
this section.

Now, we are ready to introduce the multilevel preconditioner as
follows:
\begin{equation} \label{eq:multilevel}
\tilde{B}_{\rm ad} = I_H^h B_H (I_H^h)' + \sum_{i=1}^J \iota_i
S_i^{-1} \iota_i'. 
\end{equation}
Here, $B_H: W_H' \mapsto W_H$ is the multilevel preconditioner for
$A_H$, see \cite{schoberl1999multigrid, schoberl1999multigridthesis}.
Then we have the following theorem.

\begin{theorem}\label{thm:multilevel-condition-number}
If 
$$ 
\langle B_H^{-1} w_H, w_H \rangle \eqsim \langle A_H w_H, w_H \rangle
\qquad \forall w_H \in W_H,
$$ 
then the condition number of $\tilde{B}_{\rm ad}S$ satisfies 
$$
\mathrm{cond}(\tilde{B}_{\rm ad}S) \le C (1+N_c) \frac{H^2}{\delta^2},
$$
where $C$ is independent to the mesh size $h$ and the Lam\'e constants. 
\end{theorem}
\begin{proof}
It follows directly from Theorem \ref{thm:additive-condition-number}
and norm equivalence between $\|\cdot\|_{A_H}$ and
$\|\cdot\|_{B_H^{-1}}$. 
\end{proof}

\begin{remark}
\begin{enumerate} 
\item When $H \lesssim \delta$, the preconditioners
\eqref{eq:additive-precond2} and \eqref{eq:multilevel} are both
uniform with respect to $h$ and the Lam\'{e} constants.
\item It can be proved that the corresponding multiplicative preconditioners
are also uniform as well as the additive version, provided that the local
problems associated with the subdomains are solved exactly (cf.
\cite{cho2008new, hu2013combined, antonietti2015two}). 
\item Some robust multilevel methods to solve the linear elasticity problem
can be found in \cite{brenner1996multigrid, lee2009robust,
de2013subspace, hong2016robust, chen2016fast}. By constructing stable intergrid
transfer operators similar to $I_H^h$, it is feasible to construct
corresponding multilevel solvers to the hybridized mixed method.
\end{enumerate}
\end{remark}

\subsection{Intergrid Transfer Operator $I_H^h$}
\label{subsec:intergrid}

The construction of intergrid transfer operator $I_H^h$ is divided
into two steps: (i) the intergrid transfer operator from coarse grid
to fine grid proposed by Sch{\"o}berl \cite{schoberl1999multigrid},
and (ii) the $L^2$ projection operator to Lagrange multiplier space.
More precisely, we first define the $\mathcal{P}_2$ Lagrange finite
element space $W_h$ on $\cT_h$
$$
W_h := \{ w \in H_0^1(\Omega; \R^2)~|~ w|_K \in \cP_2(K;\R^2) \quad
\forall K\in \cT_h \},
$$
with the primal elastic norm $\|\cdot\|_{A_h}$, bilinear form
$a_h(\cdot, \cdot)$, and $A_h: W_h \mapsto W_h'$ similar to
\eqref{eq:elastic-norm}. In \cite{schoberl1999multigrid}, the harmonic
extension $\tilde{I}_{H}^h: W_H \mapsto W_h$ was defined as
follows: For $w_H\in W_H$, the value of $\tilde{I}_{H}^h w_H$ on
each edge of coarse element $K_H\in \cT_H$ does not change, and the
value in the interior of $K_H$ is defined by discrete harmonic
extension, that is,
\begin{equation} \label{eq:schober-I}
\begin{aligned}
\tilde{I}_{H}^h w_H |_{\partial K_H} &= w_H|_{\partial K_H}, \\ 
a_h(\tilde{I}_H^h w_H, v_h) &= 0 \qquad \quad \forall v_h\in W_{h,0}(K_H),
\end{aligned}
\end{equation}
where $W_{h,0}(K_H) :=  \{ w\in H^1_0(K_H;\R^2)~|~ w|_{K'} \in
\cP_2(K';\R^2) \quad \forall K' \in K_H\}$. 
$\tilde{I}_H^h$ has the following stability property (cf.
\cite{schoberl1999multigrid}), 
\begin{equation}\label{eq:p2stability}
\| \tilde{I}_H^h w_H \|_{A_h} \lesssim \|w_H\|_{A_H} \qquad \forall
w_H \in W_H.
\end{equation}

The intergrid transfer operator $I_H^h$ appearing in
\eqref{eq:additive-precond2} is
defined as the product of two operators, 
\begin{equation}\label{eq:def-intergrid-operator}
I_{H}^h := Q_h
\tilde{I}_H^h: W_H \mapsto M_{h,k+1},
\end{equation} 
where $Q_h:W_h \mapsto M_{h,k+1}$ is the $L^2$ projection on edges 
(i.e., $\langle Q_h w_h , \mu \rangle_{\cF_h} := \langle  w_h ,
\mu \rangle_{\cF_h}, \forall \mu\in M_{h,k+1}$). 
Then, we have the following lemma:

\begin{lemma} \label{lm:I-stability}
The intergrid transfer operator $I_H^h: W_H \mapsto M_{h,k+1}$ has the
following stability property: 
\begin{equation}
\|I_H^h w_H \|_S \lesssim \|w_H\|_{A_H} \qquad \forall w_H \in W_H.
\label{eq:I-stability-AH}
\end{equation}
\end{lemma}

\begin{proof} 
Note that $Q_h$ is the $L^2$ projection on $M_{h,k+1}$. Then, for any
$w_h \in W_h$,
$$ 
\begin{aligned}
|Q_h w_h|_{*,K} &= |K|^{-1/2} \left| \int_{\partial K}
Q_h w_h \cdot \nu ~ds \right|  
= |K|^{-1/2} \left| \int_{\partial K}
    w_h \cdot \nu ~ds \right| \\ 
&= |K|^{-1/2} \left| \int_K \div w_h ~dx \right| =
\|P_0^h \div w_h\|_{0, K},
\end{aligned}
$$ 
and 
$$
\begin{aligned}
|Q_h w_h|_{h, K} &= \sup_{\btau \in Z_h(K)} \frac{\langle Q_h w_h,
  \btau \nu \rangle_{\partial K}}{\|\btau\|_{0,K}} =  \sup_{\btau \in
Z_h(K)} \frac{\langle  w_h, \btau \nu \rangle_{\partial
K}}{\|\btau\|_{0,K}}\\ &=  \sup_{\btau \in Z_h(K)} \frac{(
\epsilon(w_h), \btau )_{ K}}{\|\btau\|_{0,K}} \leq
\|\epsilon(w_h)\|_{0,K},
\end{aligned}
$$ 
which implies that $\|Q_h w_h\|_S \lesssim \|w_h\|_{A_h}$ due to
Theorem \ref{lm:equivalent-norm}. The stability property
\eqref{eq:I-stability-AH} then follows from \eqref{eq:p2stability} and
the stability property of $Q_h$.
\end{proof}

\subsection{Stable Decomposition} \label{subsec:stable-decom}
In this section, we shall present the stable decomposition. A key tool
to prove the stable decomposition is the interpolation  $\Pi_h:
M_{h,k+1} \mapsto W_h$, which is used to capture the low-frequency of
the multiplier $\lambda\in M_{h,k+1}$. 

We first define a parameter-independent problem: Find
$(\bar{\bsigma}_{\lambda}, \bar{u}_{\lambda}) \in  \Sigma_{h,k+1}^{-1}
\times V_{h,k}$ such that for any element $K\in \cT_h$, 
\begin{subequations}\label{eq:lsm-nopara}
\begin{align}\label{eq:lsm1-nopara}
& (\bar{\bsigma}_ {\lambda}, \btau_h )_{K} + ( \bar{u}_{\lambda}, \div
\btau_h )_{K} = \langle \lambda, \btau_h \nu\rangle_{\partial
  K} \qquad \forall \btau_h \in\cP_{k+1}(K;\mS),\\
& (\div \bar{\bsigma}_ \lambda, v_h)_{K}  \qquad\qquad\qquad = 0
\qquad \qquad \qquad \forall v_h \in \cP_{k}(K;\mathbb{R}^2). 
\label{eq:lsm2-nopara}
\end{align}
\end{subequations}
We then introduce the following rigid motion space on each element
$K$,
\begin{equation} \label{eq:RM}
{\rm RM}(K) := \{v\in H^1(K,\mathbb{R}^2)~|~(\nabla v + (\nabla v)^T)/2 =
0\}.
\end{equation}
We also introduce a projection $P_{K,{\rm RM}}: M_{h,k+1}(\partial K) \mapsto
\mathrm{RM}(K)$ by 
\begin{equation} \label{eq:proj-RM}
(P_{K, {\rm RM}}\lambda, r)_{K} = (\bar{u}_{\lambda}, r)_{K} \qquad \forall
r \in \mathrm{RM}(K).
\end{equation}
Then, the construction of the interpolation $\Pi_h$ is divided into two
steps.  First, a Cl\'ement type interpolation $\Pi_{1,h}:
M_{h,k+1}\mapsto (\cP_{1,h})^2\cap H^1(\Omega; \R^2)$ is defined as, for
any $a\in \cN_h$,
$$
(\Pi_{1,h} \lambda)(a) := 
\left\{ 
\begin{matrix} 
\frac{\sum_{K\in \omega_a} (P_{K, {\rm RM}} \lambda)(a)}{\sum_{K\in
\omega_a} 1} & a \notin \partial \Omega, \\
0 & a \in \partial \Omega,
\end{matrix}
\right. 
$$
where $\cP_{1,h}$ is the piecewise linear Lagrange element and
$\omega_a$ is the set of elements containing the vertex $a$.  Next, we
define the correction operator $\Pi_{2,h}: M_{h,k+1}\cup
H_1(\Omega,\R^2) \mapsto W_h$: 
$$
(\Pi_{2,h} \lambda) (a) := 0 \quad \forall a\in \cN_h, \quad
\text{and} \quad
\int_F \Pi_{2,h} \lambda ~ds := \int_F  \lambda ~ds \quad \forall F
\in \cF_h.
$$
Then, the interpolation $\Pi_h$ is composed by these two operators, 
\begin{equation} \label{eq:stable-interpolation}
\Pi_{h} \lambda := \Pi_{1,h}\lambda + \Pi_{2,h} (\lambda - \Pi_{1,h}\lambda)
\qquad \forall \lambda \in M_{h,k+1}.
\end{equation}
Note that the interpolation $\Pi_h$ is only used for analysis and will
not occur in the computation. To prove the stability and approximation
property of $\Pi_h$, we present some lemmas on  $P_{K, {\rm RM}}$.

\begin{lemma}\label{lm:RMstability}
It holds that 
\begin{equation}\label{eq:RMstability}
h_K^{-1}\|\bar{u}_{\lambda} - P_{K, {\rm RM}} \lambda\|_{0,K}
\lesssim|\lambda|_{h,K}.
\end{equation}
\end{lemma}
\begin{proof}
By definition, we have  $(\bar{u}_\lambda - P_{K,{\rm RM}} \lambda) \in
{\rm RM}(K)^{\bot}$.  According to the Theorem 2.2 in
\cite{hu2015finite}, we can find $\tilde{\btau}_h \in
\Sigma_{k+1,b}(K)$ such that
$$
\div \tilde{\btau}_h = \bar{u}_\lambda - P_{K,{\rm RM}} \lambda \quad
\text{and} \quad h_K^{-1}\|\tilde{\btau}_h\|_{0,K} \lesssim \|\div
\tilde{\btau}_h\|_0 = \| \bar{u}_\lambda - P_{K, {\rm RM}}\lambda \|_{0,K}. 
$$
Since $\tilde{\btau}_h \nu|_{\partial K} = 0$, \eqref{eq:lsm1-nopara}
implies that 
$$
(\bar{\bsigma}_ {\lambda}, \tilde{\btau}_h )_{K} + ( \bar{u}_{\lambda}
-P_{K,{\rm RM}}\lambda, \div \tilde{\btau}_h )_{K} = 0.
$$
Thus, 
$$
\begin{aligned}
\|\bar{u}_\lambda - P_{K,{\rm RM}} \lambda\|_{0,K}^2 &=
(\bar{u}_\lambda -P_{K,{\rm RM}} \lambda, \div \tilde{\btau}_h)_{K} =
-(\bar{\bsigma}_{\lambda}, \tilde{\btau}_h) \le
\|\bar{\bsigma}_{\lambda}\|_{0,K} \|\tilde{\btau}_h\|_{0,K} \\
& \lesssim h_K \|\bar{u}_\lambda - P_{K, {\rm RM}}\lambda\|_{0,K}
\|\bar{\bsigma}_\lambda\|_{0,K}. 
\end{aligned}
$$
Note that $\bar{\bsigma}_{\lambda} \in Z_h(K)$ by
\eqref{eq:lsm2-nopara}. By definition of $|\cdot|_{h,K}$ in
\eqref{eq:norm-h}, we have
$$
|\lambda|_{h,K} = \sup_{\btau_h \in Z_h(K)} \frac{\langle \lambda,
  \btau_h\nu\rangle_{\partial K}}{\|\btau_h\|_{0,K}} = \sup_{\btau_h
\in Z_h(K)} \frac{(\bar{\bsigma}_h, \btau_h)_K}{\|\btau_h\|_{0,K}} =
\|\bar{\bsigma}_{\lambda}\|_{0,K}.
$$
Then, \eqref{eq:RMstability} follows from the above two equations. 
\end{proof}

\begin{lemma}\label{lm:tr-rigid-stability}
Assume that $\kappa \ge\kappa_0 >0$. It holds that 
\begin{equation}\label{eq:tr-rigid-stability}
h_K^{-1}\|\lambda - P_{K, {\rm RM}} \lambda\|_{0,\partial K}^2 \lesssim
\sum_{K'\in \omega_K }  | \lambda|_{h,K'}^2.
\end{equation}
\end{lemma}
\begin{proof}
By virtue of Lemma \ref{lm:RMstability}, the triangle inequality, and
the trace inequality, we only need to prove 
\begin{equation}\label{eq:tr-rigid-stability1}
h_K^{-1}\|\lambda - \bar{u}_\lambda\|_{0,\partial K}^2 \lesssim
\sum_{K'\in \omega_K }  | \lambda|_{h,K'}^2.
\end{equation}

Consider the element patch $\omega_K$. Let $ \cF_h(\omega_K) := \cF_h
\cap \bar{\omega}_K$ and $\Sigma_{h,k+1}^{-1}(\omega_K) := \{\btau_h
\in L^2(\omega_K; \mS)~|~ \btau_h|_K' \in \cP_{k+1}(K';\mS) \quad
\forall K'\in \omega_K\}$.  By summation of \eqref{eq:lsm1-nopara}
over elements $K'\in \omega_K$, we have
\begin{equation} \label{eq:patch-equ}
(\bar{\bsigma}_ {\lambda}, \btau_h )_{\omega_K} + ( \bar{u}_{\lambda},
\div_h \btau_h )_{\omega_K} = \sum_{K' \in \omega_K}\langle \lambda,
\btau_h \nu\rangle_{\partial K'} \qquad  \forall \btau_h  \in
  \Sigma_{h,k+1}^{-1}(\omega_K).
\end{equation}
Note that $\bar{u}_{\lambda}|_{K} \in \cP_k(K;\mathbb{R}^2)$, we
denote by $\bar{u}_K$ the natural continuous extension of
$\bar{u}_{\lambda}|_{K}$ on $\omega_K$ (i.e., $\bar{u}_K$ and
$\bar{u}_\lambda|_K$ have the same polynomial form). Then, we can
recast \eqref{eq:patch-equ} as 
\begin{equation}\label{eq:patch}
(\bar{\bsigma}_ {\lambda}, \btau_h)_{\omega_K} + ( \bar{u}_{\lambda}
- \bar{u}_K, \div_h \btau_h )_{\omega_K} - (\epsilon(\bar{u}_K), \btau_h)_{\omega_K}
= \sum_{F \in \cF_h(\omega_K)}  \langle \lambda - \bar{u}_K,
[\btau_h]\rangle_{F}. 
\end{equation}

Since $\kappa \ge\kappa_0>0$, by Lemma \ref{lm:stableJ}, there exists
$\btau_1 \in \Sigma_{h,k+1}^{-1}(\omega_K)$ such that
\begin{equation} \label{eq:tau1}
[\btau_1]|_F = 
\begin{cases}
(\lambda - u_K)|_F&  F\in \partial K, \\ 
0 & \text{otherwise},
\end{cases}
\quad \text{and} \quad \|\btau_1\|_{0,\omega_K}^2 \lesssim h_K
\sin^{-2}(\kappa_0)\|\lambda - u_K\|_{0,\partial K}^2.
\end{equation}
Apply Lemma \ref{lm:infsup-high} or \ref{lm:infsup-low} on $\omega_K$,
we immediately know that there exists $\btau_2\in
\Sigma_{h,k+1}(\omega_K) := \{\btau \in H(\div,\omega_K;\mS)~|~
\btau|_{K'} \in \cP_{k+1}(K';\mS) \quad \forall K' \in \omega_K\}$
such that 
\begin{equation} \label{eq:tau2}
\div \btau_2 = -\div_h \btau_1\quad \text{and} \quad
h_K^{-1}\|\btau_2\|_{0,\omega_K} +\|\div \btau_2\|_{0,\omega_K}
\lesssim \|\div \btau_1\|_{0,\omega_K} \lesssim
h_K^{-1}\|\btau_1\|_{0,\omega_K}.
\end{equation}

Next, there is a unique decomposition that $\bar{u}_K|_K = \theta_1 +
\theta_2 \in {\rm RM}(K) \oplus {\rm RM}(K)^{\perp}$. Due to
Theorem 2.2 in \cite{hu2015finite}, we can find $\tilde{\btau}_3 \in
\Sigma_{k+1,b}(K)$ such that 
$$ 
\div \tilde{\btau}_3 = \theta_2 \quad \text{and} \quad
h_K^{-1}\|\tilde{\btau}_3\|_{0,K} \lesssim \|\theta_2\|_{0,K}.
$$ 
Let $\supp(\btau_3) \subset K$ and  
$$ 
\btau_3|_K = \begin{cases}
\frac{(\btau_1 + \btau_2,
    \epsilon(\bar{u}_K))_{\omega_K}}{\|\theta_2\|_{0,K}^2} \tilde{\btau}_3
& \theta_2 \neq 0, \\
\boldsymbol{0} & \theta_2 = 0.
\end{cases}
$$ 
A straightforward calculation shows that $\btau_3$ satisfies  
\begin{equation} \label{eq:tau3}
-(\btau_3, \epsilon(\bar{u}_K))_K = (\div\btau_3, \theta_2)_K =
(\btau_1+\btau_2, \epsilon(\bar{u}_K))_{\omega_K}, 
\end{equation}
and 
\begin{equation} \label{eq:tau3-2}
\begin{aligned}
\|\btau_3\|_{0,K} &\leq \|\btau_1 + \btau_2\|_{0,\omega_K}
\frac{\|\epsilon(\bar{u}_K)\|_{0,\omega_K}
\|\tilde{\btau}_3\|_{0,K}}{\|\theta_2\|_{0,K}^2} \\ 
& \lesssim \|\btau_1 + \btau_2\|_{0,\omega_K}
\frac{h_K\|\epsilon(\bar{u}_K)\|_{0,K}}{\|\theta_2\|_{0,K}}
\lesssim \|\btau_1+\btau_2\|_{0,\omega_K}.
\end{aligned}
\end{equation}

Thus, take $\btau = \btau_1+\btau_2 +\btau_3$ in \eqref{eq:patch},
we have
$$
[\btau]|_F = 
\begin{cases}
(\lambda - u_K)|_F & F\in \partial K,\\ 
0 & \text{otherwise},
\end{cases}
~
(\bar{u}_{\lambda} - \bar{u}_K, \div \btau)_{\omega_K}=0  \quad
\text{and} \quad (\epsilon(\bar{u}_K), \btau)_{\omega_K} = 0. 
$$
In addition, \eqref{eq:tau1}, \eqref{eq:tau2}, and \eqref{eq:tau3-2}
imply that 
$$
 \|\btau\|_{0,\omega_K}\lesssim 
 \|\btau_1\|_{0,\omega_K} \lesssim h_K^{1/2}
 \sin^{-1}(\kappa_0)\|\lambda - \bar{u}_K\|_{0,\partial K}.
$$
Hence, we have
$$
\begin{aligned}
 \|\lambda - \bar{u}_K\|_{0,\partial K}^2 &= (\bar{\bsigma}_
{\lambda}, \btau)_{\omega_K} \lesssim \|\bar{\bsigma}_
{\lambda}\|_{0,\omega_K}
 \|\btau\|_{0,\omega_K} \\ 
 &\lesssim (\sum_{K'\in \omega_K }| \lambda|_{h,K'}^2)^{1/2}  h_K^{1/2}
 \sin^{-1}(\kappa_0) \|\lambda - u_K\|_{0,\partial K},
\end{aligned}
$$ 
which gives rise to \eqref{eq:tr-rigid-stability1}.
\end{proof}

Now, we are in the place to prove the stability and approximation
property of $\Pi_h$.
\begin{lemma}\label{lm:stable-interpolation}
For any $\lambda \in M_{h,k+1}$, it holds that 
\begin{eqnarray}
 \int_{F}  \Pi_h \lambda ~ds &=& \int_{F}  \lambda ~ds \quad \forall F
 \in \cF_h, \label{eq:stable-interpolation1}\\ 
 \|\Pi_h \lambda\|_{A_h} & \lesssim& \|\lambda \|_S
 ,\label{eq:stable-interpolation2}\\
 \|\lambda - Q_h \Pi_h \lambda\|_{0}^2 &\lesssim& h
 \|{\lambda}\|_S^2.\label{eq:stable-interpolation3}
\end{eqnarray}
\end{lemma}
\begin{proof}
By the definition of $\Pi_h$ in \eqref{eq:stable-interpolation}, we have 
$$
\int_F  \lambda - \Pi_h \lambda~ds = \int_F (I - \Pi_{2,h})(I -
\Pi_{1,h}) \lambda~ds = 0,
$$
which gives rise to \eqref{eq:stable-interpolation1}.

Since $ \|\Pi_h \lambda\|_{A_h}^2 = 2\tilde{\mu} \|\epsilon(\Pi_h
\lambda)\|_0^2 + \tilde{\lambda} \|P_0^h \div(\Pi_h \lambda)\|_0^2$,
we prove the stability \eqref{eq:stable-interpolation2} of $\Pi_h$ part
by part. By \eqref{eq:stable-interpolation1}, we have
\begin{equation}\label{eq:stablity-div}
\begin{aligned}
\|P_0^h \div \Pi_h \lambda \|_{0,K} & = |K|^{-1/2} \left| \int_K \div
(\Pi_h \lambda) ~dx \right| =  |K|^{-1/2} \left| \int_{\partial K}
(\Pi_h \lambda) \cdot \nu ~ds \right| \\
&=|K|^{-1/2} \left| \int_{\partial K} \lambda \cdot \nu ~ds \right| =
|\lambda|_{*, K}.
\end{aligned}
\end{equation}

Next, we estimate $\|\epsilon(\Pi_h \lambda) \|_{0,K}$.  First, we show
the stability of $\Pi_{1,h}$ as 
\begin{equation}\label{eq:stablity-L2I1h}
\begin{aligned}
\|\Pi_{1,h}\lambda - P_{K,{\rm RM}}\lambda\|_{0,K}^2 & \lesssim h_K^n
\sum_{a\in \cN_K} | (\Pi_{1,h}\lambda)(a) - (P_{K,{\rm RM}}\lambda)(a)|^2\\
& \lesssim h_K^n \sum_{a\in \cN_K}  \sum_{\substack{\bar{K}_1 \cap
\bar{K}_2 = \bar{F} \\ \bar{F}\ni a }}| (P_{K_1,{\rm RM}}\lambda)(a) -
(P_{K_2,{\rm RM}} \lambda)(a)|^2 \\
& \lesssim h_K^n \sum_{a\in \cN_K}  \sum_{\substack{\bar{K}_1\cap
\bar{K}_2 = \bar{F} \\  \bar{F}\ni a }}\big| {\lambda}|_F(a)
-(P_{K_1,{\rm RM}}\lambda)(a) \big|^2 + \big|  {\lambda}|_F(a) -
(P_{K_2,{\rm RM}}\lambda)(a)  \big|^2\\
& \lesssim  \sum_{a\in \cN_K}  \sum_{\substack{\bar{K}_1 \cap
  \bar{K}_2 = \bar{F} \\ \bar{F} \ni a }}h_{K_1} \| {\lambda}
-P_{K_1,{\rm RM}}\lambda\|_{0,\partial K_1}^2 +h_{K_2} \| {\lambda}-
P_{K_2, {\rm RM}}\lambda \|_{0,\partial K_2}^2\\
& \lesssim  \sum_{K' \in \omega_K}  h_{K'}\| {\lambda}-
P_{K',{\rm RM}}\lambda \|_{0,\partial K'}^2. 
\end{aligned}
\end{equation}
Then, by the triangle inequality and inverse inequality, we have 
\begin{equation}\label{eq:stablity-L2Ih}
\begin{aligned}
\|\Pi_h \lambda - P_{K,{\rm RM}} \lambda\|_{0,K}^2 &\lesssim
\|\Pi_h\lambda - \Pi_{1,h}\lambda\|_{0,K}^2 +  \|\Pi_{1,h} \lambda -
P_{K,{\rm RM}}\lambda\|_{0,K}^2\\
&=   \|\Pi_{2,h} (I - \Pi_{1,h})\lambda\|_{0,K}^2 + \|\Pi_{1,h}
\lambda - P_{K,{\rm RM}}\lambda\|_{0,K}^2\\
&\lesssim h_K   \| (I - \Pi_{1,h})\lambda\|_{0,\partial K}^2 +
\|\Pi_{1,h} \lambda - P_{K,{\rm RM}} \lambda\|_{0,K}^2\\
&\lesssim   \sum_{K' \in \omega_K}  h_{K'}\| {\lambda}-
P_{K',rm}\lambda \|_{0,\partial K'}^2. \quad \text{(by
\eqref{eq:stablity-L2I1h})}
\end{aligned}
\end{equation}
Hence, by the inverse inequality, we have  
\begin{equation}\label{eq:stability-rigid}
\begin{aligned}
\|\epsilon(\Pi_h\lambda)\|_{0,K}^2 &= \|\epsilon(\Pi_h\lambda -
P_{K,{\rm RM}} \lambda ) \|_{0,K}^2\\
&\lesssim h_K^{-2} \|\Pi_h\lambda - P_{K,{\rm RM}}\lambda \|_{0,K}^2\\
&\lesssim \sum_{K' \in \omega_K}  h_{K'}^{-1}\| {\lambda}-
P_{K',{\rm RM}}\lambda \|_{0,\partial K'}^2.  \quad \qquad \text{(by
    \eqref{eq:stablity-L2Ih})} 
\end{aligned}
\end{equation}
By virtue of \eqref{eq:tr-rigid-stability}, we sum
\eqref{eq:stablity-div} and \eqref{eq:stability-rigid} over all
elements to obtain
\begin{equation*}
\|\Pi_h\lambda\|_{A_h} \lesssim \|\lambda\|_S.
\end{equation*}

Next, the approximation property \eqref{eq:stable-interpolation3}
can be obtained by summing the following inequalities over all
elements:
\begin{equation}
\begin{aligned}
\|\lambda - Q_h \Pi_h \lambda\|_{0, \partial K}^2  &\lesssim 
\|(I-Q_h)\Pi_h \lambda\|_{0,\partial K}^2 +  
\|\lambda - P_{K,{\rm RM}} \lambda  \|_{0, \partial K}^2 + \|P_{K,
{\rm RM}} \lambda - \Pi_h \lambda\|_{0, \partial K}^2 \\
&\lesssim h_K^4 |\Pi_h\lambda|_{2,\partial K}^2 + \sum_{K' \in
  \omega_K} \| {\lambda}- P_{K',{\rm RM}}\lambda \|_{0,\partial K'}^2
  \quad \text{(by \eqref{eq:stablity-L2Ih})} \\
&\lesssim h_K \|\epsilon(\Pi_h \lambda)\|_{0,K}^2 + \sum_{K' \in
  \omega_K} \| {\lambda}- P_{K',{\rm RM}}\lambda \|_{0,\partial K'}^2.
\end{aligned}
\end{equation}
This completes the proof.
\end{proof}

Let $\{W_i\}_{i=1}^J$ be the local spaces of $W_h$ associated with 
the overlapping domain decomposition $\{\Omega_i\}_{i=1}^J$. According
to \cite{schoberl1999multigridthesis}, we have the following lemma.
 
\begin{lemma}\label{lm:p2-decomposition}
For any $w_h \in W_h$, there exists a decomposition $w_h =
\tilde{I}_H^h w_H + \sum_{i=1}^J w_i $, such that $w_H\in W_H, w_i \in
W_i$, and 
\begin{equation} \label{eq:p2-decomposition}
\|w_H\|_{A_H}^2 + \sum_{i=1}^J \|w_i\|_{A_h}^2 \lesssim
\frac{H^2}{\delta^2} \|w_h\|_{A_h}^2. 
\end{equation} 
\end{lemma}

\begin{theorem} \label{lm:stable-decomposition}
For any $\lambda \in M_{h,k+1}$, there exists a decomposition $\lambda
= I_{H}^h w_H + \sum_{i=1}^J \lambda_i$ such that $w_H\in W_H,
\lambda_i \in M_i$, and  
\begin{equation} \label{eq:stable-decomposition} 
\|w_H\|_{A_H}^2 + \sum_{i=1}^J\|\lambda_i\|_{S}^2 \lesssim
  \frac{H^2}{\delta^2} \|\lambda\|_{S}^2.
\end{equation} 
\end{theorem}
\begin{proof}
We first split $\lambda$ into two components 
$$
\lambda = Q_h \underbrace{\Pi_h \lambda}_{w_h} + \underbrace{(\lambda
- Q_h\Pi_h \lambda)}_{\lambda_0}. 
$$
According to the Lemma \ref{lm:stable-interpolation}, we know that 
$$
\lambda_0 \in M_{h,0}^{\bot} \quad \text{and}\quad \|\lambda_0 \|_0^2
\lesssim h \|\lambda\|_S^2. 
$$
Here, $M_{h,0}^{\perp}$ is the $L^2$ orthogonal complement of $M_{h,0}$
in the space $M_{h,k+1}$. Denote the $L^2$ projection on
$M_{h,0}^{\perp}$ by $Q_0^{\perp}$.  Let $w_h =  \tilde{I}_H^h w_H
+\sum_{i=1}^J w_i $ be the decomposition in Lemma
\ref{lm:p2-decomposition}. We define the $\lambda_i$ as 
$$
\lambda_i =  Q_h w_i + Q_0^{\bot} (\theta_i \lambda_0) \qquad
j = 1, 2, \cdots, J.
$$
Thus, $\lambda = I_H^h w_H + \sum_{i=1}^J \lambda_i$.  By the property
of the partition of unity, Theorem \ref{lm:equivalent-norm}, Lemma
\ref{lm:I-stability}, \ref{lm:stable-interpolation}, and  
\ref{lm:p2-decomposition}, we have
$$ 
\begin{aligned}
\sum_{i=1}^J \|{\lambda}_i\|_{S}^2 &= \sum_{i=1}^J \sum_{K\in
\cT_h\cap \Omega_i} \|\lambda_i\|_{S,K}^2 \\ 
&\lesssim \sum_{i=1}^J \| Q_h w_i \|_{S}^2 + \sum_{i=1}^J \sum_{K
\in \cT_h\cap \Omega_i } \| Q_0^\perp (\theta_i \lambda_0) \|_{S,K}^2 \\ 
&\lesssim \sum_{i=1}^J \| Q_h w_i \|_{S}^2 +  
\sum_{i=1}^J \sum_{K \in \cT_h\cap \Omega_i } h_K^{-1}\|Q_0^\perp
(\theta_i \lambda_0)\|_{0,\partial K}^2  \qquad (\text{by }
    \eqref{eq:equivalent-energy-norm})\\ 
&\lesssim \sum_{i=1}^J  \|Q_h w_i \|_{S}^2 + h^{-1}\| \lambda_0
\|_{0}^2  \\ 
&\lesssim \frac{H^2}{\delta^2} \|w_h\|_{A_h}^2 +  h^{-1}\| \lambda_0
\|_{0}^2 ~ \qquad \quad \qquad (\text{by Lemma } \ref{lm:I-stability} \text{ and }
    \eqref{eq:p2-decomposition})\\ 
&\lesssim \frac{H^2}{\delta^2} \| \lambda \|_{S}^2, \qquad \qquad
\qquad \qquad \qquad \qquad \qquad (\text{by }
\eqref{eq:stable-interpolation2} \text{ and } \eqref{eq:stable-interpolation3})
\end{aligned}
$$
and
$$
\|w_H\|_{A_H} \lesssim \frac{H^2}{\delta^2}\|w_h\|_{A_h} \lesssim
\frac{H^2}{\delta^2}\|\lambda\|_S.
$$
This completes the proof.
\end{proof}

\section{Numerical Examples} \label{sec:numerical}

In this section, we give several numerical examples to present the
optimal convergence order of the hybridized mixed discretization as
well as the uniform convergence of the iterative solvers. All the
numerical experiments are implemented using the iFEM package
\cite{chen2008ifem}.

\subsection{Convergence Order Tests}
To verify the convergence order for the discretization, we take the
domain to be unit square $\Omega = (0,1)^2$ and choose the data with
the exact solution given by 
\begin{equation}\label{test_example}
u = 
\begin{pmatrix}
\re^{x-y} xy(1-x)(1-y) \\
\sin(\pi x)\sin(\pi y)
\end{pmatrix}.
\end{equation}
We apply a homogeneous boundary condition that $u = 0$ on $\partial
\Omega$. The Lam\'{e} constants are set as $\tilde{\mu} = 1/2$ and
$\tilde{\lambda} =1$.  The exact stress function $\bsigma$ and the
load function $f$ can be analytically derived from the
\eqref{equ:elasticity} for a given $u$.  We use the MATLAB backslash
solver for the system of the multiplier if the grid is singular-vertex
free, and the conjugate gradient method with diagonal preconditioning
otherwise. 

\begin{example}[Lowest order method on macro-simplex grid] 
\label{ex:lowest-order} 

Our first numerical example is carried out on the macro-simplex grid,
which can be obtained from any triangulation by connecting the
vertices of each triangle to the barycenter, thereby subdividing the
triangle into three, see Figure \ref{fig:hct}.

\begin{figure}[!htbp]
\centering
\includegraphics[scale =.6]{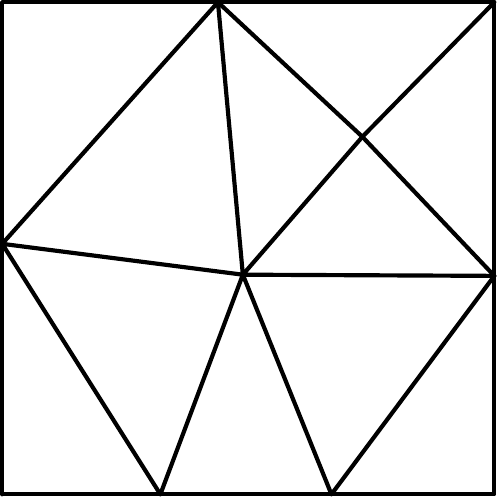}
\quad \quad \quad \quad
\includegraphics[scale =.6]{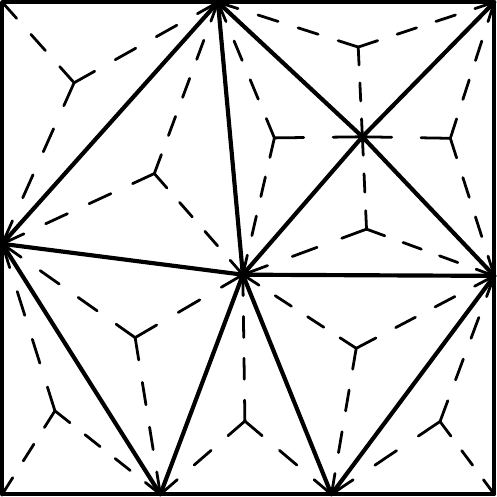}
\caption{Unstructured grid and a typical macro-simplex grid.}
\label{fig:hct} 
\end{figure}

\begin{table}[!htbp]
\caption{Errors and observed convergence orders on macro-simplex
grids, $k=0$.}\centering
{\small{
\begin{tabular}{@{}c|cc|cc|cc@{}}
  \hline
  $1/h$	&$\|u -u_h\|_0$	& Order	& 
  $\|\bsigma-\bsigma_h\|_0$ & Order & 
  $\|\div\bsigma-\div\bsigma_h\|_0$	& Order	\\
   \hline
  4		&9.5309e-2	&--	  &2.0147e-1	&--	  &2.6589e-0	&--	  \\
  8		&4.5289e-2	&1.07	&4.9971e-2	&2.01	&1.2995e-0	&1.03	\\
  16	&2.2009e-2	&1.04	&1.2357e-2	&2.01	&6.3735e-1	&1.02	\\
  32	&1.0976e-2	&1.00	&3.1761e-3	&1.96	&3.1827e-1	&1.00	\\
  64	&5.4797e-3	&1.00	&8.0961e-4	&1.97	&1.5892e-1	&1.00 
   \\ \hline
\end{tabular} 
}}
\label{numerical_result_HCT}
\end{table}

After computing \eqref{eq:hybrid-system} for various values of $h$, we
calculate the errors between the exact solution and the discrete
solution and report them in Table \ref{numerical_result_HCT}.  The
table indicates the optimal convergence orders of
$\mathcal{O}(h)$ for both stress and displacement in the  $H(\div)$
and $L^2$ norm, respectively.
\end{example}

\begin{example}[High order method]

We apply the finite element method with $k=2$ for the high order case,
which is the lowest order method that works for any 2D regular grid.
The computations are performed on both the uniform grid and crisscross
grid as depicted in Figure \ref{fig:crisscross}.

We list the errors and observed convergence orders of the computed
solution on the uniform grid in Table \ref{numerical_result_uniform}.
It clearly indicates that $\|u - u_h\|_0 = \mathcal{O}(h^3)$ and
$\|\bsigma - \bsigma_h\|_{H(\div)} = \mathcal{O}(h^3)$ which agrees with  Theorem \ref{thm:error-estimate}. In addition, we
observe that $\|\bsigma - \bsigma_h\|_0 = \mathcal{O}(h^4)$. Similar
results can be observed on the crisscross grid as shown in Table
\ref{numerical_result_singular}.  As discussed in Section
\ref{sec:hybrid}, the singular vertices do not affect the
well-posedness of the original saddle point problem but only results
in a SPSD system for the Lagrange multiplier, which can be solved
efficiently by the Krylov solvers.

\begin{figure}[!htbp]
\centering
\includegraphics[scale =.6]{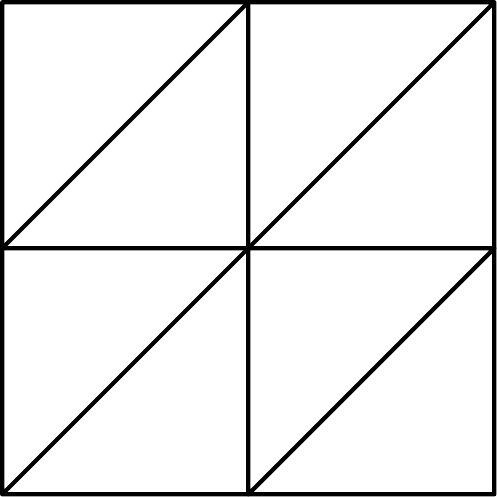}
\quad \quad \quad \quad
\includegraphics[scale =.6]{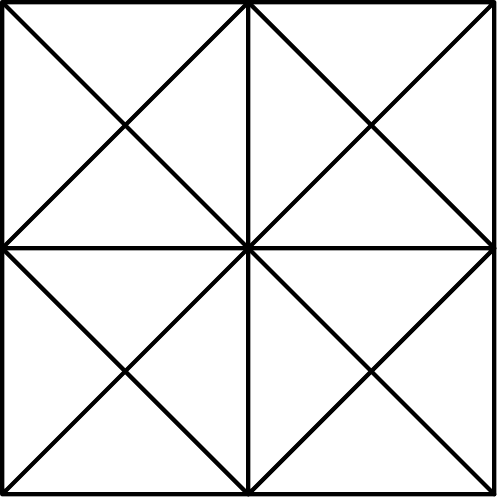}

\caption{Uniform grid and crisscross grid}
\label{fig:crisscross} 
\end{figure}

\begin{table}[!htbp]
\caption{Errors and observed convergence orders on uniform grids, $k=2$.}\centering
{\small{
\begin{tabular}{@{}c|cc|cc|cc@{}}
   \hline
  $1/h$	&$\|u -u_h\|_0$	& Order	& 
  $\|\bsigma-\bsigma_h\|_0$ & Order & 
  $\|\div\bsigma-\div\bsigma_h\|_0$	& Order\\ \hline
  4		&2.1758e-3	&-- 	&2.0260e-3	&--	  &6.2558e-2	&--	  \\
  8		&2.7561e-4	&2.98	&1.5145e-4	&3.89	&7.9274e-3	&2.98	\\
  16	&3.4569e-5	&2.99	&9.7454e-6	&3.95	&9.9431e-4	&2.99	\\
  32	&4.3248e-6	&2.99	&6.1737e-7	&3.98	&1.2439e-4	&3.00	\\
  64	&5.4072e-7	&3.00	&3.8838e-8	&3.99	&1.5552e-5	&3.00\\	
   \hline
\end{tabular}}}
\label{numerical_result_uniform}
\end{table}

\begin{table}[!htbp]
\caption{Errors and observed convergence orders on crisscross
  grids, $k=2$. }\centering
{\small{
\begin{tabular}{@{}c|cc|cc|cc@{}}
   \hline
  $1/h$	&$\|u -u_h\|_0$	& Order	& 
  $\|\bsigma-\bsigma_h\|_0$ & Order & 
  $\|\div\bsigma-\div\bsigma_h\|_0$	& Order \\ \hline
  4		&5.7633e-4		&--	  &3.1371e-4	&--	  &1.7027e-2	&-- 	\\
  8		&7.2355e-5		&2.99	&2.0057e-5	&3.96	&2.1361e-3	&2.99	\\
  16	&9.0541e-6		&2.99	&1.2672e-6	&3.98	&2.6726e-4	&2.99	\\
  32	&1.1320e-6		&3.00	&7.9629e-8	&3.99	&3.3416e-4	&3.00	\\
  64	&1.4151e-7		&3.00	&4.9899e-9	&4.00	&4.1772e-5	&3.00	\\
   \hline
\end{tabular}}}
\label{numerical_result_singular}
\end{table}

\end{example}

\subsection{Iterative Solver Tests}

In this subsection, we investigate the robustness of our iterative solvers 
with respect to both the mesh size $h$ and Poisson's
ratio $\tilde{\nu}$. In all the numerical experiments below, we
choose the data such that the exact solution is given by  
\eqref{test_example}. High-order discretization of $k=2$ is applied on
the uniform grids.  The Lam\'{e} constants are set as $\tilde{\mu} =
1/2$ and
$$
\tilde{\lambda} = \frac{\tilde{\nu}}{1-2\tilde{\nu}},
$$
where $\tilde{\nu}$ represents the Poisson's ratio that goes to $0.5$
when the material becomes increasingly incompressible. 

We run the various preconditioning Conjugate Gradient (PCG) computations with zero initial guess and a
stopping criterion whereby the relative residual is smaller than
$10^{-6}$.  We verify the reasonableness of our choices for
Schwarz smoother, intergrid transfer operator and coarse
solvers in the following numerical experiments.

\begin{example}[One-level Schwarz preconditioner]

This example is to verify the $\tilde{\nu}$-independent property of
the Schwarz method on the fine grid. Clearly any local space defined on
the vertex patch (edges that share the same vertex) belongs to one
subspace defined in \eqref{eq:local-space-M} at least. Hence, the
corresponding Schwarz method would be uniform with respective to
$\tilde{\nu}$.  We also test the other two choices of the space
decompositions with supported sets on edge patches and element
patches, respectively. 

Table \ref{tab:mini-subdomain} presents the number of iterations of
PCG with symmetrized multiplicative Schwarz preconditioner
for different decompositions.  The mesh size is set as $h=1/4$. 
Only the decomposition consisting of vertex
patches provides a $\tilde{\nu}$-independent method. 

\begin{table}[!htbp]
\caption{Number of iterations of PCG: One-level multiplicative
Schwarz preconditioner with subspaces supported on edges, elements, and
vertex patches.}\centering
{\footnotesize{
\begin{tabular}{c|cccccc@{}}
   \hline
  \backslashbox[4mm]{Subdomains} {$\tilde{\nu}$} 		&	$0.49$	&	$0.499$&
    $0.4999$&	$0.49999$&	$0.499999$&	$0.4999999$\\  \hline
  Edges 	&	36	&	59	&	79	&	109	&	131	&	154		\\
  Elements	&	15	&	24	&	33	&	45	&	54	&	62	\\
 Vertex Patches&	10	&	12	&	13	&	13	&	14	&	14\\
  \hline
\end{tabular}}}
\label{tab:mini-subdomain}
\end{table}

\end{example}

\begin{example}[Two-level Schwarz preconditioners]

We now validate the robustness of the two-level Schwarz preconditioner. We note that the $\cP_2$ Lagrange finite element space $W_H$ is used
as coarse space due to its d.o.f. that preserve rigid-body motion as well as the moments on the edges, see Lemma
\ref{lm:stable-interpolation}.  The fine grid $\cT_h=\{K_h\}$ is
refined uniformly from the coarse grid $\cT_H=\{K_H\}$.  Hence, the
overlap is set as $\delta = h$ and the ratio $H/\delta = 2$. The
intergrid transfer operator is defined as $I_H^h = Q_h \tilde{I}_H^h$
in \eqref{eq:def-intergrid-operator}.

Table \ref{table:intergrid} lists the number of iterations of PCG
using the additive Schwarz preconditioner \eqref{eq:additive-precond2}
and the corresponding symmetrized multiplicative Schwarz
preconditioner. This result clearly shows the robustness of the
Schwarz preconditioner in agreement with the Theorem
\ref{thm:additive-condition-number}. 
\begin{table}[!htbp]
\caption{Number of iterations of PCG: Two-level additive Schwarz
  preconditioner (left) and symmetrized multiplicative Schwarz
    preconditioner (right)
}\centering
{\footnotesize{
\begin{tabular}{c|cccccc@{}}
   \hline
\backslashbox{$1/h$}{$\tilde{\nu}$}&	$0.49$	&	$0.499$&	$0.4999$&	$0.49999$&	$0.499999$&	$0.4999999$\\  \hline
  $4$&	17, 3&	18, 4&	21, 4&	23, 4&	23, 4&	23, 4\\
  $8$&	17, 4&	20, 4&	25, 4&	27, 5&	28, 5&	29, 5\\
  $16$&	18, 4&	20, 4&	26, 5&	28, 5&	29, 5&	29, 5\\
  $32$&	18, 4&	20, 4&	25, 5&	27, 5&	28, 5&	29, 5\\
   \hline
\end{tabular}}}
\label{table:intergrid}
\end{table}
\end{example}

\begin{example}[Multilevel preconditioner]
We test the scalability of a multilevel preconditioner. In
this test, we use $W_H$ (i.e., continuous space of piecewise $(\cP_2)^2$) 
as the coarse space. The intergrid transfer operator and the
overlapping subdomains and are the same as those of the second
test. Instead of using an exact solver for $A_H$, we solve the coarse
problem approximately using a W-2-2 cycle in
\cite{schoberl1999multigrid}. Table  \ref{table:multilevel} shows the
uniform convergence of the multilevel symmetrized multiplicative
preconditioner.
 

\begin{table}
\caption{Number of iterations of PCG, multilevel symmetrized
  multiplicative preconditioner.}\centering
{\footnotesize{
\begin{tabular}{c|cccccc@{}}
   \hline
\backslashbox{$1/h$}{$\tilde{\nu}$}&	$0.49$	&	$0.499$&	$0.4999$&
$0.49999$&	$0.499999$&	$0.4999999$\\  \hline
  $4$&	4&	5&	5&	5&	5&	5\\
  $8$&	4&	6&	7&	7&	7&	7\\
  $16$&	5&	6&	7&	7&	7&	7\\
  $32$&	5&	6&	7&	7&	7&	7\\
   \hline
\end{tabular}}}
\label{table:multilevel}
\end{table}

\end{example}


\section{Concluding Remarks} \label{sec:concluding}
Motivated by the critical observation on the inter-element jump of the
piecewise discontinuous symmetric-matrix-valued polynomials, we
propose a family of hybridizable mixed finite elements for linear
elasticity. These methods extend the works in \cite{arnold2002mixed,
arnold2008finite, hu2015family, hu2015finite} by relaxing the
continuity of the discrete stress on the grid vertices while
preserving the symmetry and $H(\div)$ conformity in stress
approximation. By hybridization, the solution cost for our
discretization is dominated by the cost of solving the global system
of the Lagrange multiplier. To develop robust solvers, we adopt  
the Schwarz method on the fine grid and the primal method as a coarse
problem. The key to proving the uniform convergence of our iterative
solvers is the construction of the interpolation operator $I_h$, which
is stable with the approximation property (see Lemma
\ref{lm:stable-interpolation}).  The new discretization, which
preserves the physical structure of stress,  along with the robust
solver provides a new competitive approach for stress analysis in
computational structure mechanics. 

\section*{Appendix.\ Proofs of Lemmas \ref{lm:local-basis} and
  \ref{lm:stability-local-jump} }\label{sec:appendix}

\begin{proof}[Proof of Lemma \ref{lm:local-basis}]
Denote the set of all $k+1$ degree Lagrange nodes in $\cT_h$ by
$A_{h,k+1}$. For any $K \in \mathcal{T}_h$ and $a \in A_{h,k+1} \cap
\bar{K}$, let $\varphi_a^K$ be the Lagrange nodal basis in $K$, with
zero extension in $\mathcal{T}_h \setminus K$. Further, for any $F \in
\mathcal{F}_h^i$ and $a \in A_{h,k+1} \cap \bar{F}$, let $\psi_a^F$
be the dual basis of the degree $k+1$ Lagrange basis that  
$$
\langle \psi_{a'}^F, \varphi_a^K \rangle_F = \delta_{a,a'} \quad
\mbox{and} \quad \psi_{a'}^F| _{\cF_h \backslash F} = 0.
$$
For any $a\in A_{h,k+1}$, define the local spaces 
\begin{equation}\label{eq:local-subspace}
\begin{aligned}
\Sigma_{h,k+1,a}^{-1} &:= \mathrm{span} \{ \varphi_a^K T_{ij} ~|~ 1\le i
\le j \le n, \bar{K} \ni a\}, \\
M_{h,k+1,a} &:= {\rm span}\{ \psi_a^F e_i ~|~ 1\le i \le n, F\in
\mathcal{F}_h^i, \bar{F} \ni a\},
\end{aligned}
\end{equation}
where $\{e_i~|~1 \le i \le n\}$ is the basis of $\mathbb{R}^n$ and
$\{T_{ij}=\frac{1}{2}(e_ie_j^T + e_je_i^T)~|~1\le i \le j \le n\}$ is
the basis of $\mathbb{S}$. Clearly, 
$$
\Sigma_{h,k+1}^{-1} = \bigoplus_{a\in A_{h,k+1}} \Sigma_{h,k+1,a}^{-1}
\quad \text{and} \quad M_{h,k+1} = \bigoplus_{a\in A_{h,k+1}}
M_{h,k+1,a}.
$$
Moreover, if $a\neq a'$ and $\mu \in
\mathcal{C}(\Sigma_{h,k+1,a}^{-1}) \cap
\mathcal{C}(\Sigma_{h,k+1,a'}^{-1})$, then $\mu$ vanishes at all the
Lagrange nodes on the edges. This implies that $\mu = 0$, namely  
$$ 
\mathcal{C}(\Sigma_{h,k+1,a}^{-1}) \cap
\mathcal{C}(\Sigma_{h,k+1,a'}^{-1})
= \{0\} \qquad \text{if}~a \neq a'.
$$ 
Hence, we have 
\begin{equation} \label{eq:local-basis-RC}
\mathrm{R}(\mathcal{C}) = \mathcal{C}(\Sigma_{h,k+1}^{-1}) =
\bigoplus_{a\in A_{h,k+1}} \mathcal{C}(\Sigma_{h,k+1,a}^{-1}).
\end{equation} 
Therefore, $\mathrm{R}(\mathcal{C})$ has local basis since
$\mathcal{C}(\Sigma_{h,k+1,a}^{-1})$ is locally supported for any $a
\in A_{h,k+1}$.

Next, we construct the local basis for
$\mathrm{R}(\mathcal{C})^\perp$.  Let 
\begin{equation}\label{eq:singular-vertex-eq}
M_{h,k+1,a,{\perp}} := \left\{ \mu_a \in M_{h,k+1,a}~|~ \langle
\mu_a, [\btau_a]\rangle_{\cF^i_h} =0 \quad \forall
\btau_a\in \Sigma_{h,k+1,a}^{-1} \right\}. 
\end{equation}
If $a \neq a'$, we further have  
$$
\langle \mu, \mathcal{C}\btau \rangle_{\cF_h^i} = 0 \qquad \forall
\mu \in M_{h,k+1,a}, \btau \in \Sigma_{h,k+1,a'}^{-1}.
$$
Hence, we have $M_{h,k+1,a, {\perp}} \subset \mathrm{R}(\mathcal{C})
^{\perp}$ and 
$$
\mathrm{R}(\mathcal{C}) ^{\perp} = \bigoplus _{a\in A_{h,k}}
M_{h,k+1,a,{\perp}}.
$$
Therefore, the local basis $\{\psi_1, \psi_2,\cdots,\psi_{N_2}\}$ of
$\mathrm{R}(\mathcal{C}) ^{\perp}$ comes from the union of the basis
of $M_{h,k+1,a,{\perp}}$ for all $a\in A_{h,k}$.

In addition, the basis of $M_{h,k+1,a,\bot}$ can be computed locally
according to its definition \eqref{eq:singular-vertex-eq}. In
particular, $M_{h,k+1,a,\bot}$ is nontrivial for the 2D case only if
$a$ is an interior singular vertex. Thus, if there is no interior
singular vertex in $\mathcal{T}_h$, then
$\mathrm{R}(\mathcal{C})^\perp = \{0\}$, or $M_{h,k+1} =
\mathrm{R}(\mathcal{C}) = \bigoplus_{a\in A_{h,k+1}}
\mathcal{C}(\Sigma_{h,k+1,a}^{-1})$. Further, a direct calculation
shows that  
$$ 
\mathcal{C}(\Sigma_{h,k+1,a}^{-1}) = 
\begin{cases}
\mathrm{span}\{\varphi_a^Fe_i~|~ 1\leq i \leq 2, F\in \mathcal{F}_h^i,
  \bar{F} \ni a\} & a \in \bar{\cF}_h^i, \\
\{0\} & a \notin \bar{\cF}_h^i,
\end{cases}
$$ 
where $\varphi_a^F$ denotes the Lagrange nodal basis on $F$.
Therefore, we can choose a special basis of $\mathrm{R}(\mathcal{C})$
as 
\begin{equation} \label{eq:special-basis-M}
M_{h,k+1} = \bigoplus_{a \in \bar{F}_h^i \cap A_{h,k+1}}
\mathrm{span}\{\varphi_a^F e_i ~|~ 1 \leq i \leq 2, F\in \cF_h^i,
  \bar{F}\ni a\}. 
\end{equation}
The mass matrix under the special basis \eqref{eq:special-basis-M} is
the diagonal block matrix whose diagonal block entry is the local mass
matrix under the Lagrange nodal basis on $F$. Hence, the mass matrix
$\boldsymbol{M}$ is well-conditioned because the local
mass matrix is well conditioned for the Lagrange nodal basis, which
gives rise to \eqref{eq:basis-L2-stability}. This completes the proof.
\end{proof}

\begin{proof}[Proof of Lemma \ref{lm:stability-local-jump}]
In light of \eqref{eq:local-basis-RC} in the proof of Lemma
\ref{lm:local-basis}, there exists a Lagrange node $a\in A_{h,k+1}$
such that $\varphi_i \in \mathcal{C}(\Sigma_{h,k+1,a}^{-1})$. Further,
we have 
$$ 
\varphi_i = \omega \varphi_a|_{\mathcal{F}_h}, 
$$ 
where $\varphi_a$ is the Lagrange nodal basis function at the node $a$
and $\omega \in L^2(\mathcal{F};\R^2)$ is piecewise constant and
$\mathrm{supp}(\omega) \subset \{F\in \mathcal{F}_h^i~|~ \bar{F}\ni a\}$.
Next, we construct $\btau_i\in \Sigma_{h,k+1}^{-1}$ case by case
according to the location of $a$. Clearly, if $a$ is not located on
the $\bar{\mathcal{F}}_h^i$, then $\mathcal{C}(\Sigma_{h,k+1,a}^{-1})
= \{0\}$. Hence, we only need to consider the following two cases:
Internal Lagrange node on $F\in \mathcal{F}_h^i$, or vertex of
$\cT_h$. We first state a useful tool for the analysis: For any given
vectors $v, w \in \mathbb{R}^2$, there exists $T\in \mS$ such that 
\begin{equation} \label{eq:symmetry-tool}
T v = w \quad \text{and} \quad \|T\|_{l^2} \le \sqrt{2}
\frac{\|w\|_{l^2}}{\|v\|_{l^2}}.
\end{equation}
A straightforward calculation shows that $T$ in
\eqref{eq:symmetry-tool} can be chosen as 
$$ 
T = \frac{w_1}{\|v\|_{l^2}^2} 
\begin{pmatrix}
v_1 & v_2 \\
v_2 & -v_1
\end{pmatrix} +
\frac{w_2}{\|v\|_{l^2}^2}
\begin{pmatrix}
-v_2 & v_1 \\
v_1 & v_2 
\end{pmatrix}.
$$ 

\begin{figure}[htbp!]
\centering\includegraphics[scale =.7]{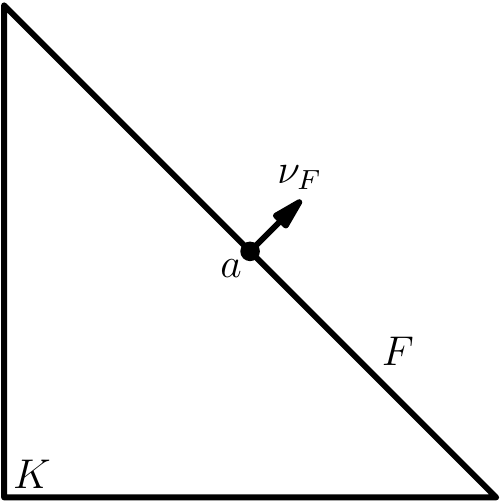}
\caption{Internal Lagrange node on edge $F$.}
\label{fig:nearly-singular1} 
\end{figure}
\paragraph{\emph{Case 1: Internal Lagrange node of $F\in
  \mathcal{F}_h^i$}} 
First, we select an element $K$ such that $F\in\bar{K}$ (cf. Figure
\ref{fig:nearly-singular1}). By virtue of \eqref{eq:symmetry-tool},
there exists $T\in \mS$ such that 
$$
T \nu_F = \omega|_F \quad \text{and} \quad \|T\|_{l^2} \lesssim
\|\omega|_F\|_{l^2}.
$$
From the definition of $\Sigma_{h,k+1,a}^{-1}$ in
\eqref{eq:local-subspace}, let $\btau_i = \varphi_a^K T \in
\Sigma_{h,k+1,a}^{-1}$. Then, 
$$
[\btau_i]_F = \varphi_i|_F \quad \forall F\in \cF_h \quad \text{and} \quad
\|\btau_i\|_{0}^2 =  \|\varphi_a\|_{0, K}^2 \|T\|_{l^2}^2\lesssim  h
\|\varphi_a\|_{0, F}^2 \|\omega|_F\|_{l^2}^2  =
h\|\varphi_i\|_{0}^2.
$$

\begin{figure}[htbp!]
\centering 
\captionsetup{justification=centering}
\subfloat[Internal vertex]{\centering 
   \includegraphics[scale=0.8]{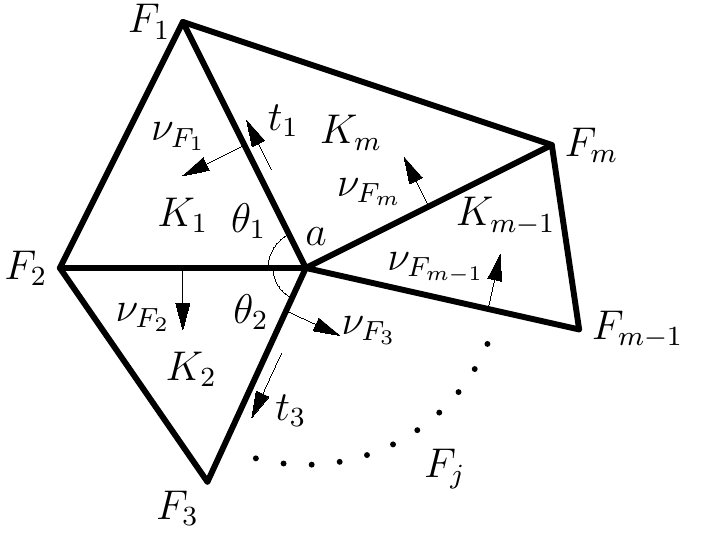} 
}%
\quad 
\subfloat[Boundary vertex]{\centering 
   \includegraphics[scale=0.8]{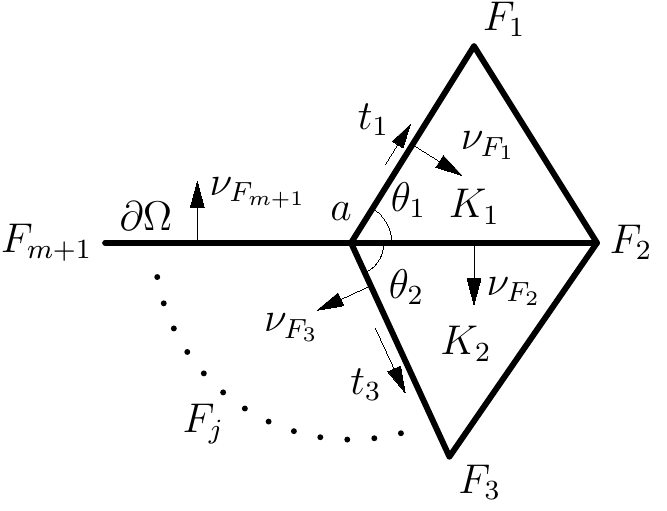} 
}
\caption{Vertex of $\mathcal{T}_h$.}
\label{fig:nearly-singular2}
\end{figure}

\paragraph{\emph{Case 2: Vertex of $\mathcal{T}_h$}} Suppose that there are
$m$ ($\ge2$) elements meeting at the vertex $a$. Since $\kappa \ge
\kappa_0>0$, there exist two adjacent elements (without loss of
generality, denoted by $K_1$ and $K_2$), such that the angles
$\theta_1$ and $\theta_2$ satisfying $|\theta_1+\theta_2-\pi|
\ge \kappa_0$, (cf. Figure \ref{fig:nearly-singular2}). The edges
that contain $a$ are denoted by $F_j$, $1\leq j \leq m$ if $a$ is an
internal vertex, and $1\leq j \leq m+1$ otherwise.  If $a$ is a
boundary vertex, we further set $F_1, F_{m+1} \in
\mathcal{F}_h^\partial$, which is feasible because
$\kappa(a) \geq \kappa_0 >0$. 

If $a$ is an internal vertex, let $F_{m+1}=F_1$ and $\nu_{F_{m+1}} =
\nu_{F_1}$. By virtue of \eqref{eq:symmetry-tool}, there exists
$T_m\in \mS$ such that 
\begin{equation} \label{eq:Tm}
T_m\nu_{F_{m+1}} = \omega|_{F_{m+1}} \quad\text{and}\quad
\|T_m\|_{l^2} \lesssim \|\omega|_{F_{m+1}}\|_{l^2}.
\end{equation}
Note that $T_m = \boldsymbol{0}\in \mS$ if $a$ is a boundary vertex.
Recursively for $j=m-1,m-2,\cdots, 2$, there exist $T_j \in \mS$ on
$K_j$ such that
\begin{equation} \label{eq:Tm-2}
T_j \nu_{F_{j+1}} =  \omega|_{F_{j+1}} + T_{j+1} \nu_{F_{j+1}}
\quad\text{and}\quad
\|T_j\|_{l^2} \lesssim   \|\omega|_{F_{j+1}}\|_{l^2} + \|T_{j+1}\|_{l^2}
\lesssim \sum_{s=j}^{m+1}  \|\omega|_{F_{s}}\|_{l^2}.
\end{equation}
Since $\omega|_{F_0} = 0$ if $a$ is a boundary vertex, we simply set
$T_1 = \boldsymbol{0}\in \mS$.

Next, we find two symmetric matrices $\tilde{T}_1 = c_1 t_1t_1^T$ and
$\tilde{T}_2 = c_2 t_3t_3^T$ on $K_1$ and $K_2$, respectively.  Here,
$t_1, t_3$ are the unit tangential vectors of $F_1$ and $F_3$,
respectively (cf. Figure \ref{fig:nearly-singular2}). The coefficients
$c_1,c_2$ are determined by 
\begin{equation} \label{eq:T-correction}
\tilde{T}_1\nu_{F_2} - \tilde{T}_2\nu_{F_2} = \omega|_{F_2} 
+ T_2 \nu_{F_2},
\end{equation}
i.e.
$$
-\big( t_1, t_3 \big)\begin{pmatrix} c_1\sin\theta_1\\
c_2\sin\theta_2 \end{pmatrix} = \omega|_{F_2} + T_2 \nu_{F_2}.
$$
Since $|\theta_1+\theta_2-\pi| \ge \kappa_0$, we have $|\det(t_1,t_3)| =
|t_1\times t_3| = |\sin(\theta_1 +\theta_2) | \ge \sin(\kappa_0)$. Thus,
the matrix $(t_1, t_3)$ is invertible. Moreover, we have $|
(t_1,t_3)^{-1} |_{\infty} \lesssim \sin^{-1}(\kappa_0) $ and, by the
shape regularity of grids, $|\sin\theta_1|$ and $|\sin\theta_2|$ are
bounded uniformly away from zero. Thus,
$$
\|\tilde{T}_1\|_{l^2}^2+\|\tilde{T}_1\|_{l^2}^2 \lesssim c_1^2+c_2^2
\lesssim \sin^{-2}(\kappa_0)  \| \omega|_{F_2} + T_2
\nu_{F_2} \|_{l^2}^2  \lesssim \sin^{-2}(\kappa_0)
\sum_{j=1}^{m+1}\|\omega|_{F_j}\|_{l^2}^2.
$$ 
In light of \eqref{eq:Tm}, \eqref{eq:Tm-2}, and
\eqref{eq:T-correction}, let 
\begin{equation} \label{eq:construction-btau-i}
\btau_i |_{K_j} =
\begin{cases}
\varphi_a^{K_j}(T_j + \tilde{T}_j) & j=1,2,\\
\varphi_a^{K_j} T_j & 3\leq j \leq m.
\end{cases}
\end{equation}
Then, we have
$$
[\btau_i]|_{F} = \varphi_i|_F \quad \forall F\in \cF_h,\quad \text{and}
\quad \|\btau_i\|_0^2 \lesssim h \sin^{-2}(\kappa_0)
  \|\varphi_i\|_{0}^2.
$$
This completes the proof.
\end{proof}


\bibliographystyle{amsplain}
\bibliography{2017hybrid}

\end{document}